\documentclass[10pt]{article}
\usepackage{amsmath,amsthm,amsfonts,amssymb}
\theoremstyle{plain}
\newtheorem{theorem}{Theorem}[section]
\newtheorem{proposition}[theorem]{Proposition}
\newtheorem{lemma}[theorem]{Lemma}
\newtheorem{corollary}[theorem]{Corollary}

\theoremstyle{definition}
\newtheorem{definition}[theorem]{Definition}

\theoremstyle{remark}
\newtheorem{remark}[theorem]{Remark}
\numberwithin{equation}{section}

\renewcommand{\eqref}[1]{(\ref{#1})}
\newcommand{\field}[1]{\mathbb{#1}}
\newcommand{\C}{\field{C}}
\newcommand{\Z}{\field{Z}}
\newcommand{\Ob}{\mathrm{Ob}}
\newcommand{\VectH}{\mathrm{Vect}_H}
\newcommand{\Hom}{\mathrm{Hom}}
\newcommand{\Rep}{\mathrm{Rep}\,}
\newcommand{\id}{\mathrm{id}}
\newcommand{\Alg}{\mathrm{Alg}}
\newcommand{\End}{\mathrm{End}}
\newcommand{\Map}{\mathrm{Map}}
\newcommand{\DR}{\mathrm{DR}}
\newcommand{\PDR}{\mathrm{PDR}}
\newcommand{\tildeotimes}{\widetilde{\otimes}}
\newcommand{\barotimes}{\bar{\otimes}}
\newcommand{\doublebarTheta}{\widetilde{\varphi}^l_3}

\newcommand{\picheck}{\check{\pi}}
\newcommand{\pibar}{\bar{\pi}}
\newcommand{\pitilde}{\widetilde{\pi}}
\begin{document}
\title{\bf FRT Construction
\\
for\\
Dynamical Yang-Baxter Maps}
\author{Youichi Shibukawa${}^a$
and Mitsuhiro Takeuchi${}^b$
\\[.4cm]
{\small ${}^a$Department of Mathematics, Faculty of Science,
Hokkaido University,}
\\
{\small Sapporo 060-0810, Japan}
\\
{\small ${}^b$Institute of Mathematics, University of Tsukuba,}
\\
{\small Tsukuba 305, Japan}
\\[.4cm]
{\small e-mail address: shibu@math.sci.hokudai.ac.jp;
takeu@math.tsukuba.ac.jp}}
\maketitle
\begin{abstract}
Notions of an $(H, X)$-bialgebroid
and of its dynamical
representation are proposed.
The dynamical representations of each $(H, X)$-bialgebroid form
a tensor category.
Every dynamical Yang-Baxter map
$R(\lambda)$ satisfying suitable conditions,
a generalization of the set-theoretical solution to the quantum
Yang-Baxter equation,
gives birth to an $(H, X)$-bialgebroid $A_R$.
The categories of L-operators for $R(\lambda)$
and of dynamical representations of $A_R$
are isomorphic as tensor categories.
\end{abstract}
\vspace*{.2cm}
\\[.3cm]
\noindent{2000 Mathematics Subject Classification(s):}
primary 16W35; secondary 16W30, 17B37, 18D10, 20G42, 81R50.
\\[.3cm]
{\allowdisplaybreaks
\section{Introduction}
Representations of every bialgebra
form a tensor category.
Faddeev, Reshetikhin, and Takhtajan
\cite{faddeev}
constructed some bialgebras,
called FRT bialgebras,
by means of
constant R-matrices,
solutions to the quantum Yang-Baxter equation
(QYBE for short)
\cite{baxter1972,baxter1982,yang1967,yang1968}.
This FRT construction is categorically characterized as follows.
The tensor category of L-operators
for the
R-matrix $R$ is isomorphic to that of representations
of the FRT bialgebra associated with $R$.

The FRT construction for the dynamical R-matrix,
a solution to the quantum dynamical Yang-Baxter equation
(QDYBE for short)
\cite{etingof2005,felder,gervais},
was developed in
\cite[Sections 4.1, 4.2, and 4.4]{etingof1998}.
This construction yields
bialgebroids \cite{bohm,lu,xu}
called dynamical quantum groups,
each of which provides 
a tensor category consisting of its dynamical representations.
This tensor category is isomorphic to that of L-operators.

In addition, the FRT construction for the Yang-Baxter map
\cite{drinfeld,lu2000,veselov,weinstein},
a set-theoretical solution
to the QYBE, was studied in \cite[Section 2.9]{etingof1999}.

The notion of a dynamical Yang-Baxter map
was introduced in previous papers
\cite{shibukawa05,shibukawa07} 
as a generalization of the Yang-Baxter map.
Let $H$ and $X$ be nonempty sets,
together with a map
$(\cdot): H\times X\ni (\lambda, x)\mapsto \lambda\cdot x\in H$.
\begin{definition}\label{intro:definition:dybm}
A map $R(\lambda): X\times X\to X\times X$ $(\lambda\in H)$ 
is a
dynamical Yang-Baxter map
associated with $H$, $X$, and $(\cdot)$,
iff
the map $R(\lambda)$ satisfies
a version of the QDYBE.
\begin{equation}\label{introduction:eq:QDYBE}
R^{23}(\lambda)R^{13}(\lambda\cdot X^{(2)})R^{12}(\lambda)
=
R^{12}(\lambda\cdot X^{(3)})R^{13}(\lambda)R^{23}
(\lambda
\cdot X^{(1)})\quad(\forall \lambda\in H).
\end{equation}
Here,
$R^{12}(\lambda)$,
$R^{12}(\lambda\cdot X^{(3)})$,
and $R^{23}(\lambda
\cdot X^{(1)})$
are the following maps
from $X\times X\times X$
to itself\/$:$
for $x, y, z\in X$,
\begin{eqnarray*}
&&
R^{12}(\lambda)(x, y, z)=(R(\lambda)(x, y), z);
\\
&&
R^{12}(\lambda\cdot X^{(3)})(x, y, z)
=
R^{12}(\lambda\cdot z)(x, y, z)
=
(R(\lambda\cdot z)(x, y), z);
\\
&&
R^{23}(\lambda
\cdot X^{(1)})(x, y, z)
=(x, R(\lambda\cdot x)(y, z)).
\end{eqnarray*}
The other maps are similarly defined.
\end{definition}

A generalization of the bijective 1-cocycle
\cite{etingof1999,lu2000,takeuchi2003}
gives birth to the dynamical Yang-Baxter map.
In \cite{shibukawa05},
we constructed dynamical Yang-Baxter maps
by means of the groups and loops
\cite[Definition I.1.10]{pflugfelder}.
Moreover,
the paper \cite{shibukawa07} established a characterization of 
the dynamical Yang-Baxter maps with some conditions
by making use of left quasigroups and
ternary operations
(See Section \ref{section:catlop}).

It is natural to try to apply the
FRT construction to the dynamical Yang-Baxter map.

In this paper, we propose notions of 
an $(H, X)$-bialgebroid
and of its dynamical representation.
The dynamical representations of each $(H, X)$-bialgebroid
form a tensor category;
the proof is based on some concepts in category theory.
We construct an $(H, X)$-bialgebroid $A_R$
associated with a bijective dynamical Yang-Baxter map
$R(\lambda)$.
This is
an analogue of the FRT construction.
We show that the tensor category
of dynamical representations of $A_R$
is isomorphic to that
of L-operators for $R(\lambda)$.

The organization of this paper is as follows.
In Section \ref{section:catvecth},
we introduce a $\C$-linear tensor category 
$\VectH$,
which plays a basic role in this paper
(cf.\ the category $\mathcal{V}_{\mathfrak h}$
\cite[Section 3.1]{etingof1998}
and the dynamical category
\cite[Section 4.2]{donin}).

We assume that
every translation map 
$\cdot x: H\ni\lambda\mapsto\lambda\cdot x\in H$
$(x\in X)$
is bijective.
This assumption produces a group
$W_X^H$ anti-isomorphic to the subgroup of $\mathrm{Aut}(H)$
generated by the set $\{ \cdot x; x\in X\}$.
In
Section \ref{section:hxalg},
we propose a notion of an $(H, X)$-algebra
(Definition \ref{hxalg:definition:hxalg}),
which is similar to that of a $W_X^H\times W_X^H$-graded
algebra.
A major difference between them is that 
the $(H, X)$-algebra is not always a direct sum of its homogeneous components
(See the end of Section \ref{section:ar}).

The category
$\Alg_{(H, X)}$ consisting of $(H, X)$-algebras
and their homomorphisms
is a pre-tensor category
(Definition \ref{hxbialg:definition:pretensorcat}),
a generalization of the tensor category.
We define an $(H, X)$-bialgebroid
as an analogue of the coalgebra object of the tensor category.
\begin{definition}[Definition \ref{hxbialg:definition:hxbialg}]
\label{introduction:definition:hxbialg}
An $(H, X)$-bialgebroid $A$ is an 
object of the category $\Alg_{(H, X)}$, together
with 
morphisms
$\Delta: A\to A\tildeotimes A$
and
$\varepsilon: A\to I_{H, X}$
of $\Alg_{(H, X)}$
such that\/$:$
\begin{eqnarray*}
&&
a^l_{A, A, A}\circ\Delta\tildeotimes\id_A\circ\Delta
=a^r_{A, A, A}\circ\id_A\tildeotimes\Delta\circ\Delta;
\\
&&
l_A\circ\varepsilon\tildeotimes\id_A\circ\Delta
=
\id_A
=
r_A\circ\id_A\tildeotimes\varepsilon\circ\Delta.
\end{eqnarray*}
\end{definition}
For the object $I_{H, X}$,
the functor $\tildeotimes$,
and
the natural transformations
$a^l$, $a^r$, $l$, $r$,
see Sections \ref{section:exphxalg} and \ref{section:hxbialg}.
This $(H, X)$-bialgebroid is a bialgebroid
(Remark \ref{hxbialg:remark:hxbialgbialg}).

We assume that the set $X$ is finite.
In Section \ref{section:ar}, we construct an $(H, X)$-bialgebroid $A_R$
associated with a bijective dynamical Yang-Baxter map $R(\lambda)$
satisfying the weight zero condition
\eqref{catlop:eq:weightzero}.

Let $D_{H, X}(V)$
$(V\in \Ob(\VectH))$
denote the $(H, X)$-algebra
defined in Section \ref{section:hxalg}
(See Proposition \ref{exphxalg:proposition:DHXV}).
By a dynamical representation
of an $(H, X)$-algebra $A$,
we mean a pair $(V, \pi)$
of an object $V$ of $\VectH$
and an $(H, X)$-algebra homomorphism
$\pi: A\to D_{H, X}(V)$
(See Section \ref{section:dynrep}).

In Section \ref{section:tpdynrep},
we first define
a crucial $(H, X)$-algebra homomorphism
\[
D_{H, X}(V)\tildeotimes D_{H, X}(W)\to
D_{H, X}(V\barotimes W)
\quad(V, W\in \Ob(\VectH)),
\]
whose definition is a bit technical, because we make use of 
seemingly infinite sum operations
(Proposition
\ref{tpdynrep:proposition:theta}).
With the aid of this canonical homomorphism,
we show
that dynamical representations of each $(H, X)$-bialgebroid $A$
form a tensor category
$\DR(A)$
(Theorem \ref{tydynrep:theorem:tensorcategory});
Section 
\ref{section:prtencat} describes its proof,
in which the pre-tensor categories,
categorical bialgebroids (Definition \ref{hxbialg:definition:catbialg}),
and pre-tensor functors (Definition \ref{prtencat:definition:pretensorfunc})
play an essential role.

The tensor category $\VectH$,
together
with a bijective dynamical
Yang-Baxter map $R(\lambda)$
satisfying the weight zero condition
\eqref{catlop:eq:weightzero}, yields
another tensor category $\Rep R$,
called a centralizer category
(cf.~\cite[Proof of Theorem 3.7]{takeuchi2005}),
consisting of L-operators for $R(\lambda)$
(See Section \ref{section:catlop}).
In the final section, Section \ref{section:isom},
we prove our main result.
\begin{theorem}[Corollary \ref{isomorph:theorem:theorem}]
\label{intro:theorem:theorem}
The tensor categories $\Rep R$ and
$\DR(A_R)$ are canonically isomorphic
with each other.
\end{theorem}

To end Introduction, it is worth pointing out that
we do not need \eqref{introduction:eq:QDYBE}
in Sections \ref{section:ar}
and \ref{section:isom}.
A role of \eqref{introduction:eq:QDYBE},
a version of the QDYBE,
is to imply that
the tensor categories
$\Rep R$
and $\DR(A_R)$
are non-trivial
in general
(See
Remark \ref{catlop:remark:nontrivial}).

\noindent{\bf Conventions.}
For a set $V$,
we will denote by
$\C V$ the $\C$-vector space whose basis is $V$.
Its element $g$ will be written as
\begin{equation}\label{intro:equation:coeff}
g=\sum_{v\in V}vg_v
\end{equation}
with complex coefficients $g_v$.
For two sets $V$ and $W$,
a $\C$-linear map
$f: \C V\to \C W$ will be identified with its matrix
$(f_{wv})_{(w, v)\in W\times V}$;
we will use the notation
\begin{equation}
\label{intro:equation:matrix}
f(v)=\sum_{w\in W}wf_{wv}
\quad(v\in V)
\end{equation}
with complex coefficients $f_{wv}$.
\section{$\C$-linear tensor category $\VectH$}
\label{section:catvecth}
Let $H$ be a nonempty set.
A $\C$-linear tensor category $\VectH$, which we will introduce
in this section,
plays a basic role in this paper.
We follow the notation of \cite[Chapter XI]{kassel}.

An object of $\VectH$ is, by definition, a pair $(V, \cdot_V)$
of a nonempty set $V$
and 
a map
$\cdot_V: H\times V\ni(\lambda, v)\mapsto \lambda\cdot_V v\in H$.
For simplicity of notation,
we write $(\cdot)$
and $V$
instead of $\cdot_V$
and $(V, \cdot_V)$,
respectively.
Furthermore,
we denote $\lambda\cdot v$
$(\lambda\in H, v\in V)$
briefly by $\lambda v$.
$\Ob(\VectH)$ is
the class of all objects of $\VectH$.

Let $V$ and $W$
be objects of $\VectH$.
A morphism $f: V\to W$ of $\VectH$
is a map $f: H\to \Hom_\C(\C V, \C W)$
such that
\begin{equation}\label{catvech:eq:defmorphism}
\lambda w=\lambda v,
\mbox{ unless }
f(\lambda)_{wv}=0
\quad(\lambda\in H, v\in V, w\in W).
\end{equation}
For the complex number $f(\lambda)_{wv}$, see
\eqref{intro:equation:matrix}.
Let $\Hom(\VectH)$ denote the class of all morphisms of $\VectH$.
For the
morphism $f: V\to W$,
the objects $V$ and $W$ are called the source 
$s(f)$ and the target $b(f)$, respectively.

For each object $V$ of $\VectH$,
define the map $\id_V: H\to \Hom_\C(\C V, \C V)$
by $\id_V(\lambda)=\id_{\C V}$.

Let $f$ and $g$ be morphisms of $\VectH$
satisfying $b(f)=s(g)$.
We define the composition 
$g\circ f\in\Hom_{\VectH}(s(f), b(g))$
by
$(g\circ f)(\lambda)=g(\lambda)\circ f(\lambda)$
$(\lambda\in H)$.
\begin{proposition}
$\VectH$ is a category
of
$\Ob(\VectH)$
and
$\Hom(\VectH)$,
together with
the identity $\id$,
the source $s$,
the target $b$,
and the composition $\circ$.
\end{proposition}

The next task is to explain a tensor product
$\barotimes: \VectH\times\VectH\to\VectH$.
Let $V$ and $W$ be objects of the category $\VectH$.
We define the set $V\barotimes W$
and the map $\cdot_{V\barotimes W}: H\times
(V\barotimes W)
\to H$
by:
\begin{eqnarray*}
&&V\barotimes W=V\times W;
\\
&&\lambda\cdot_{V\barotimes W}(v, w)
=(\lambda\cdot_{W}w)\cdot_{V}v
\quad
(\lambda \in H, v\in V, w\in W).
\end{eqnarray*}
Clearly, the pair $V\barotimes W:=(V\barotimes W,
\cdot_{V\barotimes W})$
is an object of $\VectH$.

We may identify
\begin{equation}\label{catvecth:equation:otimes}
\C(V\barotimes W)=\C V\otimes \C W
\end{equation}
by $(v, w)=v\otimes w$
$(v\in V, w\in W)$.

Let $f$ and $g$ be morphisms of the category $\VectH$.
We denote by $f\barotimes g$
the following map from $H$ to
$\Hom_\C(\C (s(f)\barotimes s(g)),
\C (b(f)\barotimes b(g)))$
(See \eqref{catvecth:equation:otimes}).
\[
f\barotimes g(\lambda)(u, v)
=
f(\lambda v)(u)\otimes
g(\lambda)(v)
\quad(\lambda\in H, u\in s(f), v\in s(g)).
\]
\begin{proposition}
$f\barotimes g\in \Hom_{\VectH}
(s(f)\bar{\otimes}s(g), b(f)\bar{\otimes} b(g))$.
\end{proposition}
\begin{proof}
It is sufficient to prove that
$f\barotimes g$ satisfies \eqref{catvech:eq:defmorphism}.
Because of \eqref{catvech:eq:defmorphism} for
the morphism $g$,
$\lambda\cdot v_1
=\lambda\cdot v_2$,
unless
$g(\lambda)_{v_1v_2}=0$;
hence,
\begin{eqnarray*}
(f\barotimes g)(\lambda)_{(u_1, v_1)(u_2, v_2)}
&=&
f(\lambda\cdot v_2)_{u_1u_2}
g(\lambda)_{v_1v_2}
\\
&=&
f(\lambda\cdot v_1)_{u_1u_2}
g(\lambda)_{v_1v_2}.
\end{eqnarray*}
On account of \eqref{catvech:eq:defmorphism} for
the morphisms $f$ and $g$,
$\lambda\cdot(u_1, v_1)
=
\lambda\cdot(u_2, v_2)$,
unless
$(f\barotimes g)(\lambda)_{(u_1, v_1)(u_2, v_2)}=0$.
This completes the proof.
\end{proof}

The associativity constraint $a: \barotimes(\barotimes\times\id)\to
\barotimes(\id\times\barotimes)$
is given by
\begin{eqnarray*}
&&a_{U, V, W}(\lambda)((u, v), w)=(u, (v, w))
\\
&&(U, V, W\in \Ob(\VectH),
\lambda\in H,
((u, v), w)\in(U\bar{\otimes}V)
\bar{\otimes}W).
\end{eqnarray*}

Let $I_{\VectH}$ denote the set consisting of one element,
together with the map
$\cdot_{I_{\VectH}}: H\times I_{\VectH}\to H$
defined by
$\lambda\cdot_{I_{\VectH}}e=\lambda$
$(\lambda\in H, I_{\VectH}=\{ e\})$.
The pair
$I_{\VectH}:=(I_{\VectH}, \cdot_{I_{\VectH}})$
is an object of $\VectH$.

The left and the right unit constraints
$l: \barotimes(I_{\VectH}\times\id)\to\id$,
$r: \barotimes(\id\times I_{\VectH})\to\id$
with respect to $I_{\VectH}$ are given by
\begin{eqnarray*}
&&
l_V(\lambda)(e, v)=r_V(\lambda)(v, e)=v
\\
&&
(V\in \Ob(\VectH), \lambda\in H, v\in V,
I_{\VectH}=\{ e\}).
\end{eqnarray*}
\begin{proposition}\label{prop:seth:sethtensor}
$(\VectH, \barotimes, I_{\VectH}, a, l, r)$
is a tensor category.
\end{proposition}

This tensor category $\VectH$ is also a $\C$-linear category.
\begin{proposition}\label{vecth:proposition:morphismsum}
Every $\Hom_{\VectH}(V, W)$ is a $\C$-vector space
with respect to\/$:$
\begin{eqnarray*}
&&(f+g)(\lambda)=f(\lambda)+g(\lambda);
\\
&&
(cf)(\lambda)=cf(\lambda)
\quad
(\lambda\in H, f, g\in \Hom_{\VectH}(V, W),
c\in\C).
\end{eqnarray*}
In addition, both the composition and the tensor product are $\C$-bilinear.
\end{proposition}

We will show some objects and morphisms of $\VectH$.
Let $V=(V, \cdot)$ be an object of $\VectH$.
\begin{proposition}\label{vecth:proposition:subobject}
For any nonempty subset $S\subset V$,
$S:=(S, \cdot_{S})$ is again an object of $\VectH$.
Here
$\cdot_S$ is the restriction of $(\cdot)$
on the set $H\times S$\/$;$
that is,
$\lambda\cdot_S s=\lambda\cdot s$
$(\lambda\in H, s\in S)$.
\end{proposition}

For $v\in V$,
we denote by
$i_V^{\{ v\}}: \{ v\}\to V$
and $p_{\{ v\}}^V: V\to \{ v\}$
the following morphisms of $\VectH$.
\begin{equation}\label{vecth:eq:injproj}
i_V^{\{ v\}}(\lambda)(v)
=v,
\quad
p_{\{ v\}}^V(\lambda)(w)
=\delta_{vw}v
\quad(\lambda\in H, w\in V).
\end{equation}
Here $\delta_{vw}$ is Kronecker's delta symbol.
These morphisms satisfy
\begin{equation}\label{catvecth:eq:projinj}
p_{\{ v'\}}^V
\circ i_V^{\{ v\}}
=
\left\{
\begin{array}{ll}
\id_{\{ v\}},&\mbox{ if }v=v';
\\
0_{\{ v'\}\{ v\}},&\mbox{ otherwise}.
\end{array}
\right.
\end{equation}
Here,
$0_{\{ v'\}\{ v\}}$
is the unit element of
$\Hom_{\VectH}(\{ v\}, \{ v'\})$.
Moreover,
if the set $V$
is finite, then
\begin{equation}\label{vecth:eq:ipidentity}
\sum_{v\in V}i_V^{\{ v\}}\circ
p_{\{ v\}}^V=\id_V.
\end{equation}

Let $V=(V, \cdot_V)$
and $W=(W, \cdot_W)$
be objects of $\VectH$
such that
$|V|=|W|=1$.
From
\eqref{catvech:eq:defmorphism},
$V$ is isomorphic to $W$,
if and only if
\begin{equation}\label{catvecth:equation:isomonepoint}
\lambda\cdot_V v=\lambda\cdot_W w
\quad(\forall\lambda\in H; V=\{ v\}, W=\{ w\}).
\end{equation}
In fact, the following $\iota_W^V: V\to W$
is an isomorphism,
if $V$ and $W$ satisfy \eqref{catvecth:equation:isomonepoint}.
\begin{equation}
\label{catvecth:equation:onepoint}
\iota_W^V(\lambda)(v)=w
\quad(\lambda\in H).
\end{equation}
\section{Centralizer categories}
\label{section:catlop}
Let $(C, \otimes, I_C, a, l, r)$
be a tensor category
\cite[Definition XI.2.1]{kassel}.
This section clarifies that
this tensor category $C$, together with an object $X_C$ and
a morphism $\sigma: X_C\otimes X_C
\to X_C\otimes X_C$,
gives birth to another tensor category
$\Rep \sigma$ consisting of L-operators for the morphism $\sigma$.
We call it a centralizer category
(cf. the center \cite[Section XIII.4]{kassel}).
Our centralizer category is isomorphic to the one in
\cite[Proof of Theorem 3.7]{takeuchi2005}
as tensor categories,
if $\sigma$ is a Yang-Baxter operator on $X_C$
(Definition \ref{catlop:definition:YBop})
in
the tensor category of vector spaces.
As a corollary of this construction, the tensor category
$\VectH$ in Section \ref{section:catvecth}
produces the centralizer category
$\Rep R$
of L-operators for the bijective dynamical
Yang-Baxter map $R(\lambda)$
satisfying the weight zero condition
\eqref{catlop:eq:weightzero}.

We will first introduce a category
$\Rep \sigma$.
An object
of $\Rep \sigma$
is, by definition, a pair $(V, L_V)$,
called an L-operator,
of
an object $V$ of $C$ and
an isomorphism
$L_V: V\otimes X_C\to
X_C\otimes V$ of $C$
satisfying 
\begin{eqnarray}
&&\nonumber
a_{X_C,X_C,V}\circ(\sigma\otimes\id_V)
\circ a_{X_C,X_C,V}^{-1}\circ(\id_{X_C}\otimes L_V)
\circ a_{X_C,V,X_C}\circ(L_V\otimes\id_{X_C})
\\\nonumber
&=&
(\id_{X_C}\otimes L_V)\circ a_{X_C,V,X_C}\circ (L_V\otimes\id_{X_C})
\circ a_{V,X_C,X_C}^{-1}\circ(\id_V\otimes\sigma)\circ
\\\label{catlop:eq:RLL=LLR}
&&
\circ a_{V,X_C,X_C}.
\end{eqnarray}
Let $\Ob(\Rep \sigma)$ denote the class of all objects
of $\Rep \sigma$.

Let $L_V=(V, L_V)$ and $L_W=(W, L_W)$
be objects of $\Rep \sigma$.
A morphism $f: L_V\to L_W$ of $\Rep \sigma$
is a morphism $f: V\to W$ of $C$
such that
\begin{equation}\label{catlop:eq:morphism}
(\id_{X_C}\otimes f)\circ L_V
=
L_W\circ (f\otimes\id_{X_C}).
\end{equation}
We write $\Hom(\Rep \sigma)$
for the class of all morphisms of $\Rep \sigma$.

For a morphism $f: L_V\to L_W$ of $\Rep \sigma$,
the objects $L_V$ and $L_W$ are called the source 
$s(f)$ and the target $b(f)$ of the
morphism $f$, respectively.

We define the identity $\id$ of $\Rep \sigma$ by
$\id_{L_V}=\id_{V}$.
The composition
$g\circ f: s(f)\to b(g)$
of morphisms $f$ and $g$ of $\Rep \sigma$
is the same as that of
the category $C$.
\begin{proposition}
$\Rep \sigma$ is a category.
\end{proposition}

The next task is to give a tensor product
$\boxtimes: \Rep \sigma\times \Rep \sigma
\to \Rep \sigma$.
Let $L_V=(V, L_V)$ and $L_W=(W, L_W)$
be objects of the category $\Rep \sigma$.
Write
\begin{eqnarray}
\nonumber
L_V\boxtimes L_W
&=&
a_{X_C,V,W}\circ (L_V\otimes\id_{W})
\circ a_{V,X_C,W}^{-1}\circ(\id_{V}
\otimes L_W)
\circ a_{V,W,X_C}
\\\label{catlop:eq:tensor}
&
\in&\Hom_{C}
((V\otimes W)\otimes X_C,
X_C\otimes(V\otimes W)).
\end{eqnarray}
To simplify notation,
we use the same symbol
$L_V\boxtimes L_W$ for the pair
$(V\otimes W, L_V\boxtimes L_W)$.

Let $f: L_V\to L_W$ and $g: L_{V'}\to L_{W'}$ be morphisms of
$\Rep \sigma$.
We set
$f\boxtimes g=f\otimes g
\in\Hom_{C}(V\otimes V',
W\otimes W')$.
Since $f$ and $g$ satisfy \eqref{catlop:eq:morphism},
$f\boxtimes g\in \Hom_{\Rep \sigma}(L_V\boxtimes L_{V'},
L_W\boxtimes L_{W'})$.

We write $L_{I_C}
=r_{X_C}^{-1}\circ l_{X_C}
:I_C\otimes X_C\to X_C\otimes I_C\in\Hom(C)$.
Here, $l_{X_C}: I_C\otimes X_C\to X_C$
and $r_{X_C}: X_C\otimes I_C\to X_C$ are the isomorphisms
given by the left and the right unit constraints $l, r$ of
$C$.
Let $I_{\Rep \sigma}$ denote
the pair $(I_C, L_{I_C})$.

The associativity constraint $a: \boxtimes(\boxtimes\times\id)
\to\boxtimes(\id\times\boxtimes)$,
and 
the left and the right unit constraints $l$, $r$ are
defined by those of the tensor category $C$:
$a_{L_U, L_V, L_W}=a_{U,V,W}$,
$l_{L_V}=l_V$,
and
$r_{L_V}=r_V$.
\begin{proposition}\label{lop:theorem}
$(\Rep \sigma,
\boxtimes,
I_{\Rep \sigma},
a, l, r)$ 
is a tensor category.
\end{proposition}
\begin{proof}
For any pair $(V, W)$ of objects of $C$,
$l_{V\otimes W}\circ a_{I_C,V,W}=l_V\otimes
\id_W$
and
$\id_V\otimes r_W\circ a_{V,W,I_C}=
r_{V\otimes W}$
(See \cite[Lemma XI.2.2]{kassel}),
which imply
\eqref{catlop:eq:RLL=LLR}
for $L_{I_C}$
and \eqref{catlop:eq:morphism} for
$l_{L_V}$ and 
$r_{L_V}$.
The rest of the proof is straightforward.
\end{proof}

Let us introduce the notion of a Yang-Baxter operator
in $C$
(See \cite[pp.323]{kassel}),
which produces an object of $\Rep \sigma$.
\begin{definition}\label{catlop:definition:YBop}
An automorphism $\sigma$ on $X_C\otimes X_C$
is a Yang-Baxter operator on $X_C$, iff
$\sigma$ satisfies
\[
(\id\barotimes\sigma)a(\sigma\barotimes\id)a^{-1}
(\id\barotimes\sigma)a
=
a(\sigma\barotimes\id)a^{-1}(\id\barotimes\sigma)
a(\sigma\barotimes\id).
\]
Here, $a=a_{X_C,X_C,X_C}$ and $\id=\id_{X_C}$.
\end{definition}
\begin{proposition}\label{catlop:proposition:equivYBop}
$(X_C, \sigma)\in\Ob(\Rep \sigma)$,
if and only if
$\sigma$ is a Yang-Baxter operator on $X_C$.
\end{proposition}

We now
explain the category $\Rep R$.
Fix a nonempty set $H$.
Let $X$ be a nonempty set,
together with a map
$(\cdot): H\times X\ni (\lambda, x)\mapsto \lambda\cdot x\in H$.
It is clear that
\begin{proposition}\label{lop:proposition:X}
$X=(X, \cdot)\in \Ob(\VectH)$.
\end{proposition}

Let $R(\lambda)$ $(\lambda\in H)$ be a bijective
dynamical Yang-Baxter map
(Definition \ref{intro:definition:dybm})
satisfying the weight zero condition
(cf.\ \cite{felder}).
\begin{equation}\label{catlop:eq:weightzero}
(\lambda\cdot v)\cdot u
=
(\lambda\cdot x)\cdot y,
\mbox{ if }
(u, v)=R(\lambda)(x, y)
\quad(\lambda\in H, x, y\in X).
\end{equation}
\begin{remark}\label{catlop:remark:weightzeroinvariance}
The weight zero condition
\eqref{catlop:eq:weightzero}
is a generalization of the invariance condition in
\cite[$(3.4)$]{shibukawa07}.
\end{remark}

This dynamical Yang-Baxter map 
$R(\lambda)$
produces a Yang-Baxter operator on $X$ in the tensor category
$\VectH$.
Let $\sigma_R$ denote the map from $H$ to
$\Hom_\C(\C X\barotimes X, \C X\barotimes X)$
defined by
\begin{equation}\label{catlop:eq:sigmar}
\sigma_R(\lambda)(x, y)=R(\lambda)(y, x)
\quad
(\lambda\in H, (x, y)\in X\barotimes X).
\end{equation}
\begin{proposition}\label{lop:proposition:sigmaR}
The above map $\sigma_R$
is a Yang-Baxter operator on $X$ in $\VectH$.
\end{proposition}

From Propositions \ref{lop:theorem}, \ref{catlop:proposition:equivYBop},
\ref{lop:proposition:X},
and
\ref{lop:proposition:sigmaR},
\begin{proposition}\label{catlop:prop:repR}
$\Rep R:=\Rep \sigma_R$
is a tensor category, together with an object
$(X, \sigma_R)$.
\end{proposition}
\begin{remark}\label{catlop:remark:nontrivial}
If $\sharp(X)\geq 2$, then
the objects $I_{\Rep R}$
and $(X, \sigma_R)$ are not isomorphic;
hence,
$\Rep R$ is non-trivial.
\end{remark}

In \cite[Section 6]{shibukawa05},
we constructed
bijective dynamical Yang-Baxter maps
satisfying the weight zero condition
in the case that $H=X$
(See Remark \ref{catlop:remark:weightzeroinvariance}).
The following is one more example,
which is useful in Section \ref{section:ar}.

Let $Q_5$ denote the set
$\{ 0, 1, 2, 3, 4\}$
of 5 elements with
the binary operation $(\cdot)$ given by
Table \ref{catlop:table:Q5}.
\begin{table}
\centering
\begin{tabular}{c|ccccc}
& 0 & 1 & 2 & 3 & 4 \\
\hline
0 & 4 & 3 & 2 & 1 & 0 \\
1 & 3 & 1 & 0 & 2 & 4 \\
2 & 0 & 2 & 3 & 4 & 1 \\
3 & 1 & 0 & 4 & 3 & 2 \\
4 & 2 & 4 & 1 & 0 & 3 \\
\end{tabular}
\caption{The binary operation $(\cdot)$ on $Q_5$\label{catlop:table:Q5}
}
\end{table}
Here, $0\cdot 2=2$.
Since each element of $Q_5$ appears once and only once
in each row and in each column of Table \ref{catlop:table:Q5},
this $Q_5$ is a quasigroup
\cite[Definition I.1.1 and Theorem I.1.3]{pflugfelder}.
\begin{definition}\label{catlop:definition:quasigroup}
A nonempty set $Q$ with a binary operation $(\cdot)$
is a quasigroup,
iff
the translation maps $a\cdot$
and $\cdot a$ on $Q$ are bijective for all $a\in Q$\/$:$
\begin{equation}\label{catlop:eq:defquasigroup}
a\cdot: Q\ni x\mapsto a\cdot x\in Q;\ 
\cdot a: Q\ni x\mapsto x\cdot a\in Q.
\end{equation}
\end{definition}

Let $\mu$ denote the ternary operation on $\Z/5\Z$
defined by
\[
\mu(a, b, c)=a-b+c
\quad(a, b, c\in \Z/5\Z).
\]
The isomorphism $\pi: Q_5\ni a\mapsto \bar{a}\in\Z/5\Z$,
together with
the quasigroup $Q_5$ and
the ternary system $(\Z/5\Z, \mu)$,
gives birth to a bijective dynamical Yang-Baxter map
$R^{Q_5}(\lambda)$
associated with $Q_5$, $Q_5$, and $(\cdot)$
satisfying the weight zero condition
\cite[Section 3]{shibukawa07}.
In fact,
the dynamical Yang-Baxter map
$R^{Q_5}(\lambda)$ is defined as follows:
for $\lambda, a, b\in Q_5$,
\begin{eqnarray*}
&&R^{Q_5}(\lambda)(a, b)=(\eta_\lambda(b)(a), \xi_\lambda(a)(b)),
\\
&&
\xi_\lambda(a)(b)=\lambda\setminus\pi^{-1}(\mu(\pi(\lambda),
\pi(\lambda\cdot a),
\pi((\lambda\cdot a)\cdot b))),
\\
&&
\eta_\lambda(b)(a)=(\lambda\cdot\xi_\lambda(a)(b))
\setminus((\lambda\cdot a)\cdot b).
\end{eqnarray*}
Here we denote by
$a\setminus: Q_5\ni b\mapsto a\setminus b\in Q_5$
$(a\in Q_5)$
the inverse of the translation map $a\cdot$
\eqref{catlop:eq:defquasigroup}
on $Q_5$.
Moreover, $R^{Q_5}(\lambda)$ satisfies the unitary condition
\cite[Section 7]{shibukawa07}
and really depends on the parameter $\lambda$:
\begin{equation}\label{catlop:eq:RQ501243}
R^{Q_5}(0)(1, 2)=(4, 3)
\mbox{ and }
R^{Q_5}(1)(1, 2)=(4, 2).
\end{equation}
\section{$(H, X)$-algebras}
\label{section:hxalg}
Let $H$ and $X$ be nonempty sets, together
with a map
$(\cdot): H\times X\ni(\lambda, x)\mapsto \lambda\cdot x\in H$.
In order to define $(H, X)$-bialgebroids,
this section is devoted to introducing $(H, X)$-algebras
(cf.\ $\mathfrak h$-algebras \cite{etingof1998,koelink2007}),
which form a category $\Alg_{(H, X)}$.

From now on we make the following assumption.
For any $x\in X$,
the translation map
\begin{equation}\label{hxalg:eq:assumebijec}
\cdot x: H\ni\lambda\mapsto\lambda\cdot x\in H
\end{equation}
is bijective.
For example,
every quasigroup $Q$
(Definition \ref{catlop:definition:quasigroup})
satisfies the assumption in the case that $H=X=Q$.

This assumption produces a group $W_X^H$
acting faithfully on the set $H$ from the right.
By the assumption,
the translation map
$\cdot x$
\eqref{hxalg:eq:assumebijec}
is an element of
the group $Aut^{a}(H)$ of all bijections on the set $H$, together
with the product $fg=g\circ f$
$(f, g\in Aut^{a}(H))$;
hence, there exists a unique homomorphism $F$
from the free
group $W_X$ on the set $X$ to $Aut^{a}(H)$
such that $F(x)(\lambda)=\lambda\cdot x$
for $\lambda\in H$ and $x\in X$
(This symbol $F$ is a temporary notation).
Let $\lambda\in H$ and $\alpha\in W_X$.
We denote by $\lambda\cdot \alpha$
the element $F(\alpha)(\lambda)\in H$.
This $(\cdot)$ is a right action of $W_X$ on $H$,
and consequently the factor group $W_X^H$ of $W_X$
by the kernel $\{ \beta\in W_X;
\lambda\cdot \beta
=\lambda \ (\forall \lambda\in H)\}$ acts faithfully
on $H$.
For simplicity of notation,
we write the induced right action of $W_X^H$
as $\lambda \alpha$ $(\lambda\in H, \alpha\in W_X^H)$.
\begin{proposition}\label{hxalg:proposition:wxh}
$W_X^H=(W_X^H, \cdot)$ is an object of $\VectH$.
\end{proposition}

Let $M_H$ denote the $\C$-algebra
of all $\C$-valued
functions on $H$.
\begin{definition}\label{hxalg:definition:hxalg}
A unital 
$\C$-algebra $A$,
together with a set
$\{ A_{\alpha, \beta}; \alpha, \beta\in W_X^H\}$
consisting of $\C$-subspaces
of $A$
and $\C$-algebra homomorphisms
$\mu_l^A, \mu_r^A: M_H\to A$,
is called an $(H, X)$-algebra, iff\/$:$
\begin{enumerate}
\item
$A=\sum_{\alpha, \beta\in W_X^H}A_{\alpha, \beta}$\/$;$
\item
$A_{\alpha, \beta}A_{\gamma, \delta}\subset
A_{\alpha\gamma, \beta\delta}$
$(\forall\alpha, \beta, \gamma, \delta\in W_X^H)$\/$;$
\item
$1_A\in A_{1, 1}$\/$;$
\item
$\mu_l^A(f), \mu_r^A(f)\in A_{1, 1}$
for any $f\in M_H$,
and
\begin{equation}
\label{hxalg:eq:commutemulmur}
\mu_l^{A}(f)\mu_r^{A}(g)=\mu_r^{A}(g)\mu_l^{A}(f)
\quad(\forall f, g\in M_H);
\end{equation}
\item
if $a\in A_{\alpha, \beta}$, then
\begin{equation}\label{hxalg:eq:bialgebroid}
a\mu_l^{A}(f)=\mu_l^{A}(T_\alpha(f))a,
\ \mbox{and}\ 
a\mu_r^{A}(f)=\mu_r^{A}(T_\beta(f))a
\quad(\forall f\in M_H).
\end{equation}
\end{enumerate}
Here, $1$ is the unit element
of the factor group $W_X^H$.
$T_\alpha(f)\in M_H$ is
defined by 
\begin{equation}\label{hxalg:eq:Ta}
T_\alpha(f)(\lambda)=f(\lambda\alpha)
\quad(\lambda\in H).
\end{equation}
\end{definition}

In this paper, $\C$-algebra homomorphisms always respect the unit element.
We note that the sum in $(1)$ is not always a direct sum
(See the end of Section \ref{section:ar}).
\begin{definition}
\label{hxlag:proposition:hxalghom}
The map $\varphi: A\to B$ is an $(H, X)$-algebra homomorphism,
iff
$\varphi$ is a $\C$-algebra homomorphism
such that
\[
\varphi(A_{\alpha,\beta})\subset B_{\alpha,\beta},
\ 
\varphi\circ\mu_l^{A}=\mu_l^{B},
\ 
\mbox{and }
\varphi\circ\mu_r^{A}=\mu_r^{B}.
\]
\end{definition}

We will denote by $\Ob(\Alg_{(H, X)})$ and $\Hom(\Alg_{(H, X)})$
the classes of all $(H, X)$-algebras and of
all $(H, X)$-algebra homomorphisms, respectively.
\begin{proposition}
$\Alg_{(H, X)}$
is a category.
\end{proposition}

We give a construction which assigns an $(H, X)$-algebra $A_{W_X^H}$
to any unital $\C$-algebra $A$,
together with
$\C$-algebra homomorphisms
$\mu_l^A, \mu_r^A: M_H\to A$
satisfying \eqref{hxalg:eq:commutemulmur}.
Let $(A_{W_X^H})_{\alpha,\beta}$
$(\alpha, \beta\in W_X^H)$
denote the set of all elements $a$ of $A$ satisfying 
\eqref{hxalg:eq:bialgebroid}.
This is a $\C$-subspace of $A$,
and
$A_{W_X^H}:=\sum_{\alpha,\beta\in W_X^H}
(A_{W_X^H})_{\alpha,\beta}$
with
$\mu_l^A$ and $\mu_r^A$
is an $(H, X)$-algebra.

The unital $\C$-algebra
$E(V):=\End_{\C}(\Map(H, \C V))$
$(V\in\Ob(\VectH)$
has the following $\C$-algebra homomorphisms
$\mu_l^{E(V)}, \mu_r^{E(V)}: M_H\to E(V)$:
\begin{eqnarray}
&&\label{hxalg:eq:muleV}
\mu_l^{E(V)}(f)(g)(\lambda)=\sum_{v\in V}
vf(\lambda v)g(\lambda)_v;
\\
&&\label{hxalg:eq:mureV}
\mu_r^{E(V)}(f)(g)(\lambda)=f(\lambda)g(\lambda)
\\
&&\nonumber
\qquad\qquad
(\lambda\in H, f\in M_H, g\in\Map(H, \C V)).
\end{eqnarray}
For $g(\lambda)_v\in\C$, see \eqref{intro:equation:coeff}.
Because these homomorphisms satisfy
\eqref{hxalg:eq:commutemulmur},
we obtain an $(H, X)$-algebra $E(V)_{W_X^H}$.

To end this section, we assign to any object $V$
of the category $\VectH$
an object $D_{H, X}(V)$
of $\Alg_{(H, X)}$,
which produces a dynamical representation in Section
\ref{section:dynrep}.

Let $\alpha$ and $\beta$ be elements of the group $W_X^H$.
We define the $\C$-linear map $\Gamma_{\alpha,\beta}^V:
\Hom_{\VectH}(V\barotimes \{\beta\}, \{\alpha\}\barotimes V)
\to\End_\C(\Map(H, \C V))$
(See Proposition \ref{vecth:proposition:morphismsum})
by 
\begin{eqnarray*}
&&\Gamma_{\alpha,\beta}^V(u)(g)(\lambda)
=
\sum_{v\in V}v\sum_{w\in V}
u(\lambda)_{(\alpha,v)(w,\beta)}
g(\lambda\beta)_{w}
\\
&&
\quad(u\in\Hom_{\VectH}
(V\barotimes \{\beta\}, \{\alpha\}\barotimes V),
g\in\Map(H, \C V),
\lambda\in H).
\end{eqnarray*}
For $\{\alpha\}$
and
$u(\lambda)_{(\alpha,v)(w,\beta)}$,
see Propositions \ref{vecth:proposition:subobject},
\ref{hxalg:proposition:wxh},
and
\eqref{intro:equation:matrix},
respectively.
\begin{proposition}\label{hxalg:proposition:unique}
$\Gamma_{\alpha,\beta}^V$ is injective.
\end{proposition}
We write
$D_{H, X}(V)_{\alpha,\beta}$
for the image of the map $\Gamma_{\alpha,\beta}^V$.

Let $u\in
\Hom_{\VectH}(V\barotimes \{\beta\}, \{\alpha\}\barotimes V)$
and
$v\in\Hom_{\VectH}(V\barotimes \{\delta\}, \{\gamma\}\barotimes V)$
$(\alpha, \beta, \gamma, \delta\in W_X^H)$.
Write $u*_Vv$
for the following morphism of $\VectH$
whose source is $V\barotimes \{\beta\delta\}$
and whose target is $\{\alpha\gamma\}\barotimes V$.
\begin{eqnarray}\nonumber
u*_V v
&=&
\iota^{\{\gamma\}\barotimes\{\alpha\}}_{\{\alpha\gamma\}}
\barotimes
\id_V
\circ
a_{\{\gamma\}\{\alpha\}V}^{-1}
\circ
\id_{\{\gamma\}}\barotimes u
\circ
a_{\{\gamma\}V\{\beta\}}
\circ
\\\label{hxalg:equation:*V}
&&
\circ
v\barotimes\id_{\{\beta\}}
\circ
a_{V\{\delta\}\{\beta\}}^{-1}
\circ
\id_V\barotimes\iota_{\{\delta\}\barotimes\{\beta\}}^{\{\beta\delta\}}.
\end{eqnarray}
For
$\iota^{\{\gamma\}\barotimes\{\alpha\}}_{\{\alpha\gamma\}}$,
see \eqref{catvecth:equation:onepoint}.

The proof of Proposition \ref{hxalg:prop:product} is straightforward.
\begin{proposition}\label{hxalg:prop:product}
$\Gamma_{\alpha,\beta}^V(u)\Gamma_{\gamma,\delta}^V(v)
=
\Gamma_{\alpha\gamma,\beta\delta}^V(u*_Vv)
$.
\end{proposition}
From this proposition,
$D_{H, X}(V)_{\alpha,\beta}D_{H, X}(V)_{\gamma,\delta}$
$\subset
D_{H, X}(V)_{\alpha\gamma,\beta\delta}$\/$;$
furthermore,
\begin{proposition}\label{exphxalg:proposition:DHXV}
$D_{H, X}(V):=\sum_{\alpha,\beta\in W_X^H}
D_{H, X}(V)_{\alpha,\beta}$ with
$\mu_l^{D_{H, X}(V)}:=\mu_l^{E(V)}$
$\eqref{hxalg:eq:muleV}$
and
$\mu_r^{D_{H, X}(V)}:=\mu_r^{E(V)}$
$\eqref{hxalg:eq:mureV}$
is an object of $\Alg_{(H, X)}$.
\end{proposition}

$D_{H, X}(V)_{\alpha,\delta}$
is relevant to
$D_{H, X}(V)_{\beta,\gamma}$
$(\alpha, \beta, \gamma, \delta\in W_X^H)$.
Let $f$ be an element of $M_H$.
If $\alpha, \beta\in W_X^H$
satisfy that $\lambda\alpha=\lambda\beta$
unless $f(\lambda)= 0$,
then
we can
define the morphism
$f_{\alpha,\beta}: \{\beta\}\to\{\alpha\}$
of $\VectH$
by
$f_{\alpha,\beta}(\lambda)(\beta)=f(\lambda)\alpha$
$(\lambda\in H)$.

Let $g$ be an element of $M_H$.
We assume that $\gamma, \delta(\in W_X^H)$
satisfy that
$\lambda\gamma=\lambda\delta$
unless
$g(\lambda)= 0$.
A simple computation shows 
\begin{proposition}\label{hxalg:prop:point}
For $u\in
\Hom_{\VectH}(V\barotimes\{\gamma\}, \{\beta\}\barotimes V)$,
\[
\Gamma_{\alpha,\delta}^{V}
(f_{\alpha,\beta}\barotimes\id_V\circ u\circ
\id_V\barotimes g_{\gamma,\delta})
=
\mu_l^{D_{H, X}(V)}(f)\mu_r^{D_{H, X}(V)}(g)
\Gamma_{\beta,\gamma}^V(u).
\]
\end{proposition}
\begin{remark}
What we have discussed and we will discuss in the following
holds by replacing the group $W_X^H$
with an arbitrary group $G$
acting on $H$ from the right.
In this case the notion in Definition \ref{hxalg:definition:hxalg}
will be called a $G$-algebra,
and they form a category $\Alg_G$.
Just as Proposition \ref{exphxalg:proposition:DHXV},
we can construct a $G$-algebra $D_G(V)$ for any object $V$ of
$\VectH$.
\end{remark}
\section{Properties of $\Alg_{(H, X)}$}
\label{section:exphxalg}
In this section,
we will proceed with the study of the category $\Alg_{(H, X)}$,
which is essential in defining $(H, X)$-bialgebroids.

Let us first explain an object
$I_{H, X}$ of $\Alg_{(H, X)}$,
which is relevant to the object
$D_{H, X}(I_{\VectH})$
(See Propositions \ref{prop:seth:sethtensor} and 
\ref{exphxalg:proposition:DHXV}).
For $f\in M_H$,
we define
$f^\star\in\End_\C(M_H)$
by
\begin{equation}\label{exphxalg:eq:fhat}
f^\star(g)(\lambda)=f(\lambda)g(\lambda)
\quad(g\in M_H, \lambda\in H).
\end{equation}
Let $(I_{H, X})_{\alpha,\beta}$
$(\alpha, \beta\in W_X^H)$
denote the following $\C$-subspace of
$\End_\C(M_H)$.
\[
(I_{H, X})_{\alpha,\beta}
=
\left\{
\begin{array}{ll}
\{ f^\star T_\alpha;
f\in M_H\},
&
\mbox{if }
\alpha=\beta,
\\
\{ 0\},
&
\mbox{otherwise}.
\end{array}
\right.
\]
For $T_\alpha$, see \eqref{hxalg:eq:Ta}.
We set $I_{H, X}=\sum_{\alpha,\beta\in W_X^H}
(I_{H, X})_{\alpha,\beta}$,
which is a $\C$-subalgebra of $\End_\C(M_H)$.
We define the maps
$\mu_l^{I_{H, X}}, \mu_r^{I_{H, X}}:
M_H\to(I_{H, X})_{1,1}$
by
\[
\mu_l^{I_{H, X}}(f)
=\mu_r^{I_{H, X}}(f)
=f^\star T_{1}.
\]
\begin{proposition}\label{exphxalg:proposition:IHX}
The $\C$-algebra $I_{H, X}$,
together with
$\{(I_{H, X})_{\alpha,\beta};
\alpha, \beta\in W_X^H\}$,
$\mu_l^{I_{H, X}}$,
and $\mu_r^{I_{H, X}}$,
is an object of $\Alg_{(H, X)}$
$($See Definition $\ref{hxalg:definition:hxalg})$.
\end{proposition}

We can
construct a morphism
$\varphi_0: I_{H, X}\to D_{H, X}(I_{\VectH})$
of $\Alg_{(H, X)}$.
Define the $\C$-linear isomorphism
$\bar{\varphi}_0$
from $M_H$ to $\Map(H, \C I_{\VectH})$
by $\bar{\varphi}_0(f)(\lambda)=f(\lambda)e$
$(I_{\VectH}=\{ e\}, f\in M_H, \lambda\in H)$.
For $u\in\End_\C(M_H)$, we set
$\varphi_0(u)=\bar{\varphi}_0\circ u\circ(\bar{\varphi}_0)^{-1}$.
This $\C$-linear isomorphism 
$\varphi_0$
satisfies that
$\varphi_0(I_{H, X})\subset
D_{H, X}(I_{\VectH})$.
To simplify notation,
we continue to write
$\varphi_0$ for
$\varphi_0 |_{I_{H, X}}$.
\begin{proposition}\label{exphxalg:proposition:iota}
$\varphi_0: I_{H, X}\to D_{H, X}(I_{\VectH})\in\Hom(\Alg_{(H, X)})$.
\end{proposition}

Let $A$ and $B$ be objects of $\Alg_{(H, X)}$.
The next task is to
define an object
$A\tildeotimes B$ of $\Alg_{(H, X)}$,
called a
matrix tensor product
(cf.~\cite[Section 4]{etingof1998}),
which produces a functor
$\tildeotimes: \Alg_{(H, X)}\times\Alg_{(H, X)}\to\Alg_{(H, X)}$.

Write
$I_2=I_2(A, B)$
for the right ideal of $A\otimes_\C B$
generated by the elements 
\begin{equation}\label{exp:eq:I2}
\{\mu_r^{A}(f)\otimes 1_B
-1_A\otimes\mu_l^{B}(f);
f\in M_H\}.
\end{equation}
Let $\alpha$ and $\beta$ be elements of the group
$W_X^{H}$.
We denote by $(A\tildeotimes B)_{\alpha,\beta}$
the $\C$-subspace of $A\otimes_\C B/I_2$
generated by
$
\{ \xi_1\otimes\xi_2+I_2;
\xi_1\in A_{\alpha,\gamma},
\xi_2\in B_{\gamma,\beta}
\ (\exists\gamma\in W_X^H)\}$.
We set
$A\tildeotimes B=
\sum_{\alpha, \beta\in W_X^{H}}
(A\tildeotimes B)_{\alpha,\beta}
(\subset (A\otimes_\C B)/I_2)$.

For $x=\xi_x+I_2,
y=\xi_y+I_2\in A\tildeotimes B$,
define a product on $A\tildeotimes B$
by $xy=\xi_x\xi_y+I_2$.
Although $I_2$ is a right ideal,
this definition of the product makes sense.

The maps $\mu_l^{A\tildeotimes B},
\mu_r^{A\tildeotimes B}: M_H\to
(A\tildeotimes B)_{1,1}$ are given as follows.
For $f\in M_H$,
$\mu_l^{A\tildeotimes B}(f)
=\mu_l^{A}(f)\otimes 1_B+I_2$,
and
$\mu_r^{A\tildeotimes B}(f)
=1_A\otimes\mu_r^{B}(f)+I_2$.
By taking account of
\eqref{hxalg:eq:commutemulmur}
and \eqref{exp:eq:I2},
these maps are 
$\C$-algebra 
homomorphisms from $M_H$ to $A\tildeotimes B$.
\begin{proposition}\label{exphxalg:proposition:AtildeotimesB}
$A\tildeotimes B$ with
$\{ (A\tildeotimes B)_{\alpha,\beta}; 
\alpha, \beta\in W_X^H\}$,
$\mu_l^{A\tildeotimes B}$,
$\mu_r^{A\tildeotimes B}$,
and the unit $1_{A\tildeotimes B}=1_A\otimes 1_B+I_2$
is an object of $\Alg_{(H, X)}$.
\end{proposition}
\begin{remark}
On account of \eqref{hxalg:eq:bialgebroid}
and \eqref{exp:eq:I2},
the matrix tensor product $A\tildeotimes B$ of $(H, X)$-algebras
$A$ and $B$ is a $\C$-subalgebra
of the $\C$-algebra
$A\times_{M_H}B$
\cite{takeuchi}
(See \cite[Remark 2.1]{koelink2007}).
\end{remark}

Two morphisms
$f: A\to C$ and
$g: B\to D$ of $\Alg_{(H, X)}$
induce a
morphism
$f\tildeotimes g: A\tildeotimes B\to C\tildeotimes D$
of $\Alg_{(H, X)}$.
Let $f\tildeotimes g$ denote the $\C$-linear map
from $A\otimes_\C B/I_2(A, B)$
to 
$C\otimes_\C D/I_2(C, D)$
defined by
\[
(f\tildeotimes g)(\xi+I_2(A, B))
=
(f\otimes g)(\xi)+I_2(C, D)
\quad
(\xi\in A\otimes_\C B).
\]
Here, $I_2(A, B)\subset A\otimes_\C B$
is the right ideal generated by the elements
\eqref{exp:eq:I2}.
Since
$(f\otimes g)(I_2(A, B))\subset
I_2(C, D)$,
the definition of the map $f\tildeotimes g$
is unambiguous.
Because $f\tildeotimes g
((A\tildeotimes B)_{\alpha,\beta})\subset
(C\tildeotimes D)_{\alpha,\beta}$
for all $\alpha, \beta\in W_X^H$,
we use the same symbol $f\tildeotimes g$
for the $\C$-linear map 
$(f\tildeotimes g)|_{A\tildeotimes B}: A\tildeotimes B
\to C\tildeotimes D$.
\begin{proposition}\label{exphxalg:proposition:tensor}
The map $f\tildeotimes g$ is a morphism of $\Alg_{(H, X)}$.
Moreover,
$\tildeotimes: \Alg_{(H, X)}\times\Alg_{(H, X)}\to\Alg_{(H, X)}$
is a functor.
\end{proposition}

Let $A$, $B$, and $C$
be objects of $\Alg_{(H, X)}$.
The final task of this section is to introduce
a functor
$-\tildeotimes -\tildeotimes -: \Alg_{(H, X)}\times\Alg_{(H, X)}
\times\Alg_{(H, X)}\to\Alg_{(H, X)}$.
We first define
$(H, X)$-algebras
$A\tildeotimes B\tildeotimes C^l$
and $A\tildeotimes B\tildeotimes C^r$,
whose definitions are different from those of
the $(H, X)$-algebras
$(A\tildeotimes B)\tildeotimes C$
and $A\tildeotimes (B\tildeotimes C)$, respectively.

We use the symbol $I_3^l=I_3^l(A, B, C)$
for the right ideal
of $(A\otimes_\C B)\otimes_\C C$
generated by
the elements
\begin{eqnarray}
&&\nonumber
\{ (\mu_r^{A}(f)\otimes 1_B)\otimes 1_C
-(1_A\otimes\mu_l^{B}(f))\otimes 1_C; f\in M_H\}
\\
&&\cup
\{
(1_A\otimes\mu_r^{B}(g))\otimes 1_C
-(1_A\otimes 1_B)\otimes\mu_l^{C}(g); g\in M_H\}.
\label{exphxalg:eq:ABCl}
\end{eqnarray}

Let $\alpha$ and $\beta$ be elements of the group
$W_X^{H}$.
Let 
$(A\tildeotimes B\tildeotimes C^l)_{\alpha,\beta}$
denote the $\C$-subspace of 
$(A\otimes_\C B)\otimes_\C C/I_3^l$
generated by
$\{ (\xi_1\otimes\xi_2)\otimes \xi_3+I_3^l;
\xi_1\in A_{\alpha,\gamma},
\xi_2\in B_{\gamma,\delta},
\xi_3\in C_{\delta,\beta}
\ 
(\exists \gamma, \delta\in W_X^H)
\}$.
Set
$A\tildeotimes B\tildeotimes C^l=
\sum_{\alpha, \beta\in W_X^{H}}
(A\tildeotimes B\tildeotimes C^l)_{\alpha,\beta}
(\subset (A\otimes_\C B)\otimes_\C C/I_3^l)$.

For $x=\xi_x+I_3^l,
y=\xi_y+I_3^l\in A\tildeotimes B\tildeotimes C^l$,
we define a product on $A\tildeotimes B\tildeotimes C^l$
by $xy=\xi_x\xi_y+I_3^l$.
This product is well defined,
though $I_3^l$ is a right ideal.

We denote by $\mu_l^{A\tildeotimes B\tildeotimes C^l}$
and
$\mu_r^{A\tildeotimes B\tildeotimes C^l}$
the following maps from $M_H$ to
$(A\tildeotimes B\tildeotimes C^l)_{1,1}$.
For $f\in M_H$:
\[
\mu_l^{A\tildeotimes B\tildeotimes C^l}(f)
=(\mu_l^{A}(f)\otimes 1_B)\otimes 1_C+I_3^l;
\ 
\mu_r^{A\tildeotimes B\tildeotimes C^l}(f)
=(1_A\otimes 1_B)\otimes\mu_r^{C}(f)+I_3^l.
\]
Because of 
\eqref{hxalg:eq:commutemulmur}
and \eqref{exphxalg:eq:ABCl},
these maps are
$\C$-algebra homomorphisms from $M_H$ to
$A\tildeotimes B\tildeotimes C^l$.
\begin{proposition}
The $\C$-algebra $A\tildeotimes B\tildeotimes C^l$,
together with
$\{ (A\tildeotimes B\tildeotimes C^l)_{\alpha,\beta}; 
\alpha, \beta\in W_X^H\}$,
$\mu_l^{A\tildeotimes B\tildeotimes C^l}$,
$\mu_r^{A\tildeotimes B\tildeotimes C^l}$,
and the unit
$1_{A\tildeotimes B\tildeotimes C^l}
=(1_A\otimes 1_B)\otimes 1_C+I_3^l$,
is an object of $\Alg_{(H, X)}$.
\end{proposition}

The definition of the $(H, X)$-algebra
$A\tildeotimes B\tildeotimes C^r$ is similar.
Let $I_3^r$ denote the right ideal of 
$A\otimes_\C (B\otimes_\C C)$
generated by
\begin{eqnarray*}
&&
\{ \mu_r^{A}(f)\otimes (1_B\otimes 1_C)
-1_A\otimes(\mu_l^{B}(f)\otimes 1_C); f\in M_H\}
\\
&&\cup
\{
1_A\otimes(\mu_r^{B}(g)\otimes 1_C)
-1_A\otimes (1_B\otimes\mu_l^{C}(g)); g\in M_H\}.
\end{eqnarray*}
For $\alpha, \beta\in W_X^H$,
we write
$(A\tildeotimes B\tildeotimes C^r)_{\alpha,\beta}$
for the $\C$-subspace of 
$A\otimes_\C (B\otimes_\C C)/I_3^r$
generated by
$\{ \xi_1\otimes(\xi_2\otimes \xi_3)+I_3^r;
\xi_1\in A_{\alpha,\gamma},
\xi_2\in B_{\gamma,\delta},
\xi_3\in C_{\delta,\beta}
\ 
(\exists \gamma, \delta\in W_X^H)
\}$;
in addition, set
$A\tildeotimes B\tildeotimes C^r=
\sum_{\alpha, \beta\in W_X^{H}}
(A\tildeotimes B\tildeotimes C^r)_{\alpha,\beta}$.

For $x=\xi_x+I_3^r,
y=\xi_y+I_3^r\in A\tildeotimes B\tildeotimes C^r$,
we define a product on
$A\tildeotimes B\tildeotimes C^r$
by $xy=\xi_x\xi_y+I_3^r$.

We denote by $\mu_l^{A\tildeotimes B\tildeotimes C^r}$
and 
$\mu_r^{A\tildeotimes B\tildeotimes C^r}$
the following maps from $M_H$ to
$(A\tildeotimes B\tildeotimes C^r)_{1,1}$.
For $f\in M_H$:
\[\mu_l^{A\tildeotimes B\tildeotimes C^r}(f)
=\mu_l^{A}(f)\otimes(1_B\otimes 1_C)+I_3^r;
\
\mu_r^{A\tildeotimes B\tildeotimes C^r}(f)
=1_A\otimes(1_B\otimes\mu_r^{C}(f))+I_3^r.
\]
\begin{proposition}
The $\C$-algebra $A\tildeotimes B\tildeotimes C^r$, together
with
$\{ (A\tildeotimes B\tildeotimes C^r)_{\alpha,\beta}; 
\alpha, \beta\in W_X^H\}$,
$\mu_l^{A\tildeotimes B\tildeotimes C^r}$,
$\mu_r^{A\tildeotimes B\tildeotimes C^r}$,
and the unit
$1_{A\tildeotimes B\tildeotimes C^r}
=1_A\otimes (1_B\otimes 1_C)+I_3^r$,
is an object of $\Alg_{(H, X)}$.
\end{proposition}

Let $\bar{\Gamma}$ denote the $\C$-linear isomorphism from
$A\otimes_\C(B\otimes_\C C)$ to
$(A\otimes_\C B)\otimes_\C C$
satisfying
$\bar{\Gamma}(a\otimes(b\otimes c))
=(a\otimes b)\otimes c$.
This isomorphism $\bar{\Gamma}$
implies
an isomorphism $\Gamma:
A\tildeotimes B\tildeotimes C^r\to
A\tildeotimes B\tildeotimes C^l$
of $\Alg_{(H, X)}$.
Here and subsequently,
we let $A\tildeotimes B\tildeotimes C$ stand for
the $(H, X)$-algebra
$A\tildeotimes B\tildeotimes C^l
(=\Gamma(A\tildeotimes B\tildeotimes C^r))$.

Let
$f: A\to A'$, $g: B\to B'$, and $h: C\to C'$
be morphisms of $\Alg_{(H, X)}$.
Define
the map
$(f\tildeotimes g\tildeotimes h)^{\widetilde{\ }}:
(A\otimes_\C B)\times C\to (A'\otimes_\C B')\otimes_\C C'/I_3^l{}'$
by
\[
(f\tildeotimes g\tildeotimes h)^{\widetilde{\ }}(w, c)=
(f\otimes g)(w)\otimes h(c)+I_3^l{}'
\quad(w\in A\otimes_\C B, c\in C).
\]
Here,
$I_3^l{}'=I_3^l(A', B', C')$
\eqref{exphxalg:eq:ABCl}.

This map is $\C$-bilinear; for this reason,
there exists a unique $\C$-linear map
$(f\tildeotimes g\tildeotimes h)^{\bar{\ }}:
(A\otimes_\C B)\otimes_\C C
\to (A'\otimes_\C B')\otimes_\C C'/I_3^l{}'$
such that
$(f\tildeotimes g\tildeotimes h)^{\bar{\ }}(w\otimes c)
=
(f\tildeotimes g\tildeotimes h)^{\widetilde{\ }}(w, c)$.
We denote by $f\tildeotimes g\tildeotimes h$
the following map
from
$(A\otimes_\C B)\otimes_\C C/I_3^l$
to $(A'\otimes_\C B')\otimes_\C C'/I_3^l{}'$:
$f\tildeotimes g\tildeotimes h(\xi+I_3^l)=
(f\tildeotimes g\tildeotimes h)^{\bar{\ }}(\xi)
\quad(\xi\in (A\otimes_\C B)\otimes_\C C)$.
Here, $I_3^l=I_3^l(A, B, C)$.
By Definition
\ref{hxlag:proposition:hxalghom},
$(f\tildeotimes g\tildeotimes h)^{\bar{\ }}(I_3^l)\subset\{ I_3^l{}'\}$,
and the definition of the map
$f\tildeotimes g\tildeotimes h$
makes sense as a result.

The map
$f\tildeotimes g\tildeotimes h$
satisfies
$f\tildeotimes g\tildeotimes h
((A\tildeotimes B\tildeotimes C)_{\alpha,\beta})
\subset (A'\tildeotimes B'\tildeotimes C')_{\alpha,\beta}$
$(\alpha, \beta
\in W_X^H)$.
By abuse of notation,
we continue to write $f\tildeotimes g\tildeotimes h$
for the restriction
$f\tildeotimes g\tildeotimes h|_{A\tildeotimes B\tildeotimes C}:
A\tildeotimes B\tildeotimes C\to A'\tildeotimes B'\tildeotimes C'$.
\begin{proposition}\label{hxbialg:proposition:-tilde-tilde-}
$-\tildeotimes -\tildeotimes -: \Alg_{(H, X)}\times\Alg_{(H, X)}
\times\Alg_{(H, X)}\to\Alg_{(H, X)}$
is a functor.
\end{proposition}
\section{$(H, X)$-bialgebroids}
\label{section:hxbialg}
This section is devoted to introducing 
$(H, X)$-bialgebroids
(cf.\ the $\mathfrak h$-bialgebroid
\cite[Section 4.1]{etingof1998}
and the
$\times_{M_H}$-bialgebra \cite[Definition 4.5]{takeuchi});
the definition is based on 
a notion of a pre-tensor category,
which is important in the proof of
Theorem \ref{tydynrep:theorem:tensorcategory}.
\begin{definition}\label{hxbialg:definition:pretensorcat}
A pre-tensor category is a category $C$,
together with functors
$\tildeotimes: C\times C\to C$,
$-\tildeotimes -\tildeotimes -: C\times C\times C\to C$,
an object $I$ of $C$,
and
natural transformations
$a^l: \tildeotimes(\tildeotimes\times\id)\to
-\tildeotimes -\tildeotimes -$,
$a^r: \tildeotimes(\id\times \tildeotimes)\to
-\tildeotimes -\tildeotimes -$,
$l: \tildeotimes(I\times\id)\to\id$,
$r: \tildeotimes(\id\times I)\to\id$.
We denote it by $(C, \tildeotimes, -\tildeotimes -\tildeotimes -,
I, a^l, a^r, l, r)$.
\end{definition}

The pre-tensor category is a generalization of the tensor category.
\begin{proposition}\label{hxbialg:proposition:alghxpretensor}
The category $\Alg_{(H, X)}$ is a pre-tensor category.
\end{proposition}

Taking Propositions
\ref{exphxalg:proposition:IHX},
\ref{exphxalg:proposition:tensor},
and
\ref{hxbialg:proposition:-tilde-tilde-}
into account, we need only construct natural
transformations
$a^l$, $a^r$, $l$, and $r$
for the proof.

Let $A$, $B$, and $C$ be objects of $\Alg_{(H, X)}$.
Let us introduce the temporary notation $F$ for
the $\C$-bilinear map
from
$A\otimes_\C B\times C$ to 
$(A\otimes_\C B)\otimes_\C C$
defined by
$F(x, c)=x\otimes c$
$(x\in A\otimes_\C B, c\in C)$.
We write $\bar{F}$ for
the following map
from $A\otimes_\C B/I_2(A, B)\times C$
to $((A\otimes_\C B)\otimes_\C C)/I_3^l$.
\[
\bar{F}(x+I_2(A, B), c)=F(x, c)+I_3^l
\quad(x\in A\otimes_\C B).
\]
For $I_2(A, B)$ and $I_3^l=I_3^l(A, B, C)$,
see \eqref{exp:eq:I2}
and \eqref{exphxalg:eq:ABCl}.
The definition of this map
$\bar{F}$ is unambiguous,
because $F(I_2(A, B), C)\subset I_3^l$.
By abuse of notation,
we use the same symbol $\bar{F}$
for the map $\bar{F}|_{A\tildeotimes B\times C}:
A\tildeotimes B\times C\to
((A\otimes_\C B)\otimes_\C C)/I_3^l$.
Since $\bar{F}$ is $\C$-bilinear,
there uniquely exists a $\C$-linear map
$\bar{\Phi}_1: (A\tildeotimes B)\otimes_\C C
\to ((A\otimes_\C B)\otimes_\C C)/I_3^l$
such that
$\bar{\Phi}_1(x\otimes c)
=\bar{F}(x, c)$
for $x\in A\tildeotimes B$ and $c\in C$.

This $\bar{\Phi}_1$ yields the $\C$-linear map
$a^l_{A, B, C}: (A\tildeotimes B)\otimes_\C C/
I_2(A\tildeotimes B, C)
\to ((A\otimes_\C B)\otimes_\C C)/I_3^l$
defined by
$a^l_{A, B, C}(x+I_2(A\tildeotimes B, C))
=\bar{\Phi}_1(x)$
$(x\in (A\tildeotimes B)\otimes_\C C)$.
Because the map $a^l_{A, B, C}$ satisfies
$a^l_{A, B, C}((A\tildeotimes B)\tildeotimes C)\subset
A\tildeotimes B\tildeotimes C$
$(=A\tildeotimes B\tildeotimes C^l)$,
we continue to write
$a^l_{A, B, C}$
for the $\C$-linear map
$a^l_{A, B, C}|_{(A\tildeotimes B)\tildeotimes C}:
(A\tildeotimes B)\tildeotimes C\to
A\tildeotimes B\tildeotimes C$.
\begin{proposition}\label{hxbialg:proposition:al}
The map
$a^l_{A, B, C}$ is a morphism of $\Alg_{(H, X)}$.
Furthermore,
$a^l: \tildeotimes(\tildeotimes\times\id)\to
-\tildeotimes -\tildeotimes -$
is a natural transformation.
\end{proposition}

We will use the temporary notation
$\Phi_2': A\tildeotimes(B\tildeotimes C)\to
A\tildeotimes B\tildeotimes C^r$ for the morphism
of $\Alg_{(H, X)}$
defined similarly to $a^l_{A, B, C}$:
\begin{eqnarray*}
&&\Phi_2'
((a\otimes(b\otimes c+I_2(B, C)))
+I_2(A, B\tildeotimes C))
=
a\otimes (b\otimes c)+I_3^r
\\
&&
(a\in A_{\alpha,\gamma},
b\in B_{\gamma,\delta}, c\in C_{\delta,\beta},
\alpha, \beta, \gamma, \delta\in W_X^H).
\end{eqnarray*}
Let $a^r_{A, B, C}$ denote the morphism
$\Gamma\circ\Phi_2':
A\tildeotimes(B\tildeotimes C)\to
A\tildeotimes B\tildeotimes C$
of $\Alg_{(H, X)}$.
For the isomorphism $\Gamma: A\tildeotimes B\tildeotimes C^r
\to A\tildeotimes B\tildeotimes C$, see
the previous section.
\begin{proposition}\label{hxbialg:proposition:ar}
This
$a^r: \tildeotimes(\id\times\tildeotimes)\to
-\tildeotimes -\tildeotimes -$
is a natural transformation.
\end{proposition}

What is left is to introduce 
natural transformations
$l: \tildeotimes(I_{H, X}\times\id)
\to \id$
and
$r: \tildeotimes(\id\times I_{H, X})
\to \id$
(For $I_{H, X}$, see Proposition
\ref{exphxalg:proposition:IHX}).

Let $A$ be an object of $\Alg_{(H, X)}$.
We write $\bar{l}_A$ (resp.~$\bar{r}_A$)
for the following map
from $I_{H, X}\times A$ to $A$
(resp.~from $A\times I_{H, X}$ to $A$).
For $U\in I_{H, X}$ and $a\in A$,
\[
\bar{l}_A(U, a)=\mu_l^{A}(U(1_{M_H}))a
\quad
(\mbox{resp.}\ 
\bar{r}_A(a, U)=\mu_r^{A}(U(1_{M_H}))a).
\]
Here, $1_{M_H}\in M_H$ is the unit element of the $\C$-algebra $M_H$:
\begin{equation}\label{hxbialg:eq:1MH}
1_{M_H}(\lambda)=1
\quad(\lambda\in H).
\end{equation}
Because $\bar{l}_A$ (resp.~$\bar{r}_A$) is $\C$-bilinear,
there exists a unique $\C$-linear map
$\widetilde{l}_A: I_{H, X}\otimes_\C A\to A$
(resp.~
$\widetilde{r}_A: A\otimes_\C I_{H, X}\to A$)
such that $\widetilde{l}_A(U\otimes a)
=\bar{l}_A(U, a)$
(resp.~
$\widetilde{r}_A(a\otimes U)
=\bar{r}_A(a, U)$).
\begin{lemma}\label{bialg:lemma:tildeI}
$\widetilde{l}_A(I_2(I_{H, X},
A))=\widetilde{r}_A(I_2(A, I_{H, X}))
=\{ 0\}$.
\end{lemma}
\begin{proof}
We will only show that
$\widetilde{l}_A(I_2(I_{H, X},
A))\subset\{ 0\}$.
It suffices to prove
\begin{eqnarray}\label{bialg:eq:tildeI}
&&
\widetilde{l}_A((\mu_r^{I_{H, X}}(f)\otimes 1_A
-1_{I_{H, X}}\otimes\mu_l^{A}(f))(U\otimes a))=0
\\
&&\nonumber
(f\in M_H, U\in I_{H, X}, a\in A).
\end{eqnarray}
By the definition of $\widetilde{l}_A$,
\begin{equation}\label{bialg:eq:tildeI2}
\mbox{LHS of \eqref{bialg:eq:tildeI}}
=
\mu_l^{A}((\mu_r^{I_{H, X}}(f)U)(1_{M_H}))a
-
\mu_l^{A}(U(1_{M_H}))\mu_l^{A}(f)a.
\end{equation}
Simple calculation shows that
$(\mu_r^{I_{H, X}}(f)U)(1_{M_H})=U(1_{M_H})f$.
Since
$\mu_l^{A}$ is a $\C$-algebra homomorphism
by Definition \ref{hxalg:definition:hxalg},
(RHS of \eqref{bialg:eq:tildeI2})$=0$.
This is the desired conclusion.
\end{proof}

We define the $\C$-linear maps $l_A:
I_{H, X}\otimes_\C A/I_2
(I_{H, X}, A)\to A$
and
$r_A:
A\otimes_\C I_{H, X}/I_2
(A, I_{H, X})\to A$
by
$l_A(\xi+I_2
(I_{H, X}, A))=\widetilde{l}_A(\xi)$
$(\xi\in I_{H, X}\otimes_\C A)$
and
$r_A(\xi'+I_2
(A, I_{H, X}))=\widetilde{r}_A(\xi')$
$(\xi'\in A\otimes_\C I_{H, X})$.
Lemma
\ref{bialg:lemma:tildeI}
makes these definitions allowable. 
For simplicity of notation, we continue to write
$l_A$ and $r_A$
for the $\C$-linear maps
$l_A|_{I_{H, X}\tildeotimes A}: I_{H, X}\tildeotimes A
\to A$
and
$r_A|_{A\tildeotimes I_{H, X}}: A\tildeotimes I_{H, X}
\to A$, respectively.
\begin{proposition}\label{hxalg:proposition:lr}
$l_A, r_A\in\Hom(\Alg_{(H, X)})$.
Moreover,
$l: \tildeotimes(I_{H, X}\times\id)\to\id$
and
$r: \tildeotimes(\id\times I_{H, X})\to\id$
are natural transformations.
\end{proposition}

The proof of
Proposition
\ref{hxbialg:proposition:alghxpretensor}
is therefore complete;
that is,
$\Alg_{(H, X)}$, together with
$\tildeotimes, -\tildeotimes -\tildeotimes -,
I_{H, X}, a^l, a^r, l, r$, is a pre-tensor category.

A categorical bialgebroid of the pre-tensor category is an analogue of the coalgebra object 
of the tensor category.
\begin{definition}\label{hxbialg:definition:catbialg}
A categorical bialgebroid $A$ of a pre-tensor category $C$
is an object of $C$, together with
morphisms
$\Delta: A\to A\tildeotimes A$,
called coproduct,
and
$\varepsilon: A\to I$,
called counit,
of $C$
such that\/$:$
\begin{eqnarray}
&&\label{hxbialg:eq:catbialg1}
a^l_{A, A, A}\circ\Delta\tildeotimes\id_A\circ\Delta
=a^r_{A, A, A}\circ\id_A\tildeotimes\Delta\circ\Delta;
\\
&&\label{hxbialg:eq:catbialg2}
l_A\circ\varepsilon\tildeotimes\id_A\circ\Delta
=
\id_A
=
r_A\circ\id_A\tildeotimes\varepsilon\circ\Delta.
\end{eqnarray}
\end{definition}

This definition agrees with
the $\times_B$-bialgebra
\cite[Definition 4.5]{takeuchi},
which is equivalent to the bialgebroid
\cite{brzezinski};
in fact, $\times_B$-bialgebra
is exactly a categorical bialgebroid
of the pre-tensor category
$\Alg_{B\otimes \bar{B}}$
consisting of algebras over $B\otimes\bar{B}$,
together with functors
$\times_B$, $-\times_B-\times_B-$,
an object $\End(B)$,
and
natural transformations
$\alpha$, $\alpha'$,
$\theta'$, $\theta$
(For notations, see
\cite[Section 3]{takeuchi}).
\begin{definition}\label{hxbialg:definition:hxbialg}
An $(H, X)$-bialgebroid is a categorical bialgebroid
of the pre-tensor category $\Alg_{(H, X)}$.
\end{definition}
\begin{remark}
\label{hxbialg:remark:hxbialgbialg}
The $(H, X)$-bialgebroid is a bialgebroid
\cite{lu}:
the coproduct is exactly the same as
$\Delta: A\to A\tildeotimes A\subset A\otimes_\C A/I_2(A, A)$;
the counit $\varepsilon'$ is defined by
$\varepsilon'(a)=\varepsilon(a)(1_{M_H})\in M_H$
$(a\in A)$;
the left and the right actions of $M_H$ on
$A$
are respectively given by the multiplications of 
$\mu_l^{A}(f)$ and $\mu_r^{A}(f)$ $(f\in M_H)$
from the left.
\end{remark}
\section{$(H, X)$-bialgebroid $A_R$}
\label{section:ar}
In this section, we construct
an $(H, X)$-bialgebroid
$A_\sigma$
associated with
a morphism $\sigma: X\barotimes X\to X\barotimes X$ of $\VectH$
(cf. \cite{etingof1998}).
As a result,
every
bijective
dynamical Yang-Baxter map
$R(\lambda)$
satisfying the weight zero condition
\eqref{catlop:eq:weightzero}
gives birth to
the $(H, X)$-bialgebroid
$A_R:=A_{\sigma_R}$
(For $\sigma_R$, see 
Proposition \ref{lop:proposition:sigmaR}).

For the definition, we require the set $X$ to be finite;
and we will assume that $X$ is finite, whenever
we deal with the $(H, X)$-bialgebroid $A_\sigma$.

Let $HX$ denote the set
\begin{equation}\label{ar:eq:HX}
M_H\otimes_\C M_H\sqcup\{ L_{ab}; a, b\in X\}
\sqcup\{ (L^{-1})_{ab}; a, b\in X\}.
\end{equation}
Here $M_H$ is the $\C$-algebra of all $\C$-valued functions on $H$.
Moreover, we write $\C\{ HX\}$ for
the free $\C$-algebra
on the set $HX$
\cite[Section I.2]{kassel}.
\begin{definition}\label{ar:definition:ar}
$A_\sigma$ is the quotient of $\C\{ HX\}$
by the two-sided ideal $I_\sigma$
whose generators are the following $(1)$--$(5)$.
\begin{enumerate}
\item
$\xi+\xi'-(\xi+\xi')$
$(\xi, \xi'\in M_H\otimes_\C M_H)$,

$c\xi-(c\xi)$
$(c\in\C, \xi\in M_H\otimes_\C M_H)$,

$\xi\xi'-(\xi\xi')$
$(\xi, \xi'\in M_H\otimes_\C M_H)$.

Here the symbol $+$ in $\xi+\xi'$ means the addition in
the algebra $\C\{ HX\}$,
while
the symbol $+$ in $(\xi+\xi')(\in HX)$ is the addition in
the algebra $M_H\otimes_\C M_H$.
The notations of the scalar products and products
in the other generators
are similar.
\item
$\sum_{c\in X}L_{ac}(L^{-1})_{cb}-\delta_{ab}\emptyset$,

$\sum_{c\in X}(L^{-1})_{ac}L_{cb}-\delta_{ab}\emptyset$
$(a, b\in X)$.

Here $\emptyset(\in\C\{ HX\})$ and $\delta_{ab}$
mean the empty word
and Kronecker's delta symbol, respectively.
\item
$(T_{[a]}(f)\otimes 1_{M_H})L_{ab}-L_{ab}(f\otimes 1_{M_H})$,

$(1_{M_H}\otimes T_{[b]}(f))L_{ab}-L_{ab}(1_{M_H}\otimes f)$,

$(f\otimes 1_{M_H})(L^{-1})_{ab}
-(L^{-1})_{ab}(T_{[b]}(f)\otimes 1_{M_H})$,

$(1_{M_H}\otimes f)(L^{-1})_{ab}-(L^{-1})_{ab}(1_{M_H}\otimes T_{[a]}(f))$
$(a, b\in X, f\in M_H)$.

For $T_{[a]}(f)$ and $1_{M_H}$,
see
$\eqref{hxalg:eq:Ta}$
and $\eqref{hxbialg:eq:1MH}$, respectively.
\item
$\sum_{x, y\in X}(\sigma_{ac}^{xy}\otimes 1_{M_H})L_{yd}L_{xb}
-
\sum_{x, y\in X}(1_{M_H}\otimes \sigma_{xy}^{bd})L_{cy}L_{ax}$
$(a, b, c, d\in X)$.

Here $\sigma_{ac}^{xy}\in M_H$ is defined by
$\sigma_{ac}^{xy}(\lambda)=\sigma(\lambda)_{(a, c) (x, y)}$
(See \eqref{intro:equation:matrix}).
\item
$\emptyset-1_{M_H}\otimes 1_{M_H}$.
\end{enumerate}
\end{definition}
The sums in the generators $(2)$ and $(4)$ make sense,
because the set $X$ is finite.

Let $\widetilde{HX}$ denote the set of all words
$x_1x_2\cdots x_n$ in the alphabet $HX(\ni x_i)$
including the empty word $\emptyset$.
We define the map $w_l: \widetilde{HX}\to W_X^H$ 
as follows
(For $W_X^H$, see Section \ref{section:hxalg}):
if $x\in HX$, then
\[
w_l(x)=
\left\{
\begin{array}{ll}
1,&x\in M_H\otimes_\C M_H,
\\  
{[a]},&x=L_{ab},
\\
{[b^{-1}]},&x=(L^{-1})_{ab};
\end{array}
\right.
\]
if $w=x_1x_2\cdots x_n\in \widetilde{HX}$, then
$w_l(w)=w_l(x_1)w_l(x_2)\cdots w_l(x_n)$;
and we set $w_l(\emptyset)=1$.

The map  $w_r: \widetilde{HX}\to W_X^H$
is similarly defined:
if $x\in HX$, then
\[
w_r(x)=
\left\{
\begin{array}{ll}
1,&x\in M_H\otimes_\C M_H,
\\  
{[b]},&x=L_{ab},
\\
{[a^{-1}]},&x=(L^{-1})_{ab};
\end{array}
\right.
\]
if $w=x_1x_2\cdots x_n\in \widetilde{HX}$, then
$w_r(w)=w_r(x_1)w_r(x_2)\cdots w_r(x_n)$;
and we write $w_r(\emptyset)=1$.

Let $\alpha$ and $\beta$ be elements of $W_X^H$.
We write $\C\{ HX\}_{\alpha, \beta}$
and
$(A_\sigma)_{\alpha, \beta}$
for the subspaces of $\C\{ HX\}$ 
and
$A_\sigma$
generated by
the sets
$\{ w; w\in \widetilde{HX},
w_l(w)=\alpha, w_r(w)=\beta\}$
and
$\{ w+I_\sigma; w\in \widetilde{HX},
w_l(w)=\alpha, w_r(w)=\beta\}$,
respectively.
Obviously, $\C\{ HX\}=\oplus_{\alpha,\beta\in W_X^H}
\C\{ HX\}_{\alpha,\beta}$, while
the algebra $A_\sigma$ is not always a direct sum of the
components
$\{ (A_\sigma)_{\alpha, \beta};
\alpha,
\beta\in W_X^H\}$.
For more details, see 
the end of this section.

The maps $\mu_l^{A_\sigma}, \mu_r^{A_\sigma}: M_H\to
(A_\sigma)_{1,1}$ are defined as follows.
For $f\in M_H$,
\begin{equation}\label{ar:eq:mulmur}
\mu_l^{A_\sigma}(f)=f\otimes 1_{M_H}+I_\sigma;
\mu_r^{A_\sigma}(f)=1_{M_H}\otimes f+I_\sigma.
\end{equation}
\begin{proposition}
$A_\sigma$, together with $\{ (A_\sigma)_{\alpha, \beta};
\alpha, \beta\in W_X^H\}$,
$\mu_l^{A_\sigma}$, and
$\mu_r^{A_\sigma}$,
is an object of $\Alg_{(H, X)}$
$($See Definition $\ref{hxalg:definition:hxalg})$.
\end{proposition}

The next task is to define the coproduct $\Delta:
A_\sigma\to A_\sigma\tildeotimes A_\sigma$.
Let $\Delta_{HX}$ denote
the following map from
$HX$ to 
$A_\sigma\otimes_\C A_\sigma$:
\begin{eqnarray*}
&&
\Delta_{HX}(\xi)
=
\mu_l^{A_\sigma}\otimes\mu_r^{A_\sigma}(\xi)
\quad(\xi\in M_H\otimes_\C M_H);
\\
&&
\Delta_{HX}(L_{ab})
=
\sum_{c\in X}
L_{ac}+I_\sigma\otimes L_{cb}+I_\sigma;
\\
&&
\Delta_{HX}((L^{-1})_{ab})
=\sum_{c\in X}(L^{-1})_{cb}+I_\sigma\otimes (L^{-1})_{ac}+I_\sigma.
\end{eqnarray*}
This map uniquely induces a $\C$-algebra homomorphism
$\bar{\Delta}: \C\{ HX\}
\to A_\sigma\otimes_\C A_\sigma$
such that
$\bar{\Delta}(v)=\Delta_{HX}(v)$
$(v\in HX)$.
\begin{proposition}\label{ar:proposition:DIRI2}
$\bar{\Delta}(I_\sigma)\subset I_2=I_2(A_\sigma, A_\sigma)$
$($See $\eqref{exp:eq:I2}$$)$.
\end{proposition}
\begin{proof}
Since the right ideal $I_2$ satisfies that
$\bar{\Delta}(\C\{ HX\})I_2\subset I_2$,
it suffices to prove that
$\bar{\Delta}(x)\in I_2$
for every generator $x$ in Definition \ref{ar:definition:ar}.

If $x$ is one of the generators
$(1)$--$(3)$ and $(5)$,
then 
$\bar{\Delta}(x)=0_{A_\sigma\otimes_\C A_\sigma}\in I_2$.
Because
$\sum_{x, y\in X}(\sigma_{ac}^{xy}\otimes 1_{M_H})L_{yd}L_{xb}
-\sum_{x, y\in X}(1_{M_H}\otimes
\sigma_{xy}^{bd})L_{cy}L_{ax}\in I_\sigma$,
the definitions of $\bar{\Delta}$,
$\mu_l^{A_\sigma}$, and
$\mu_r^{A_\sigma}$ induce that
\begin{eqnarray*}
&&\bar{\Delta}(\sum_{x, y\in X}(\sigma_{ac}^{xy}\otimes 1_{M_H})
L_{yd}L_{xb}
-\sum_{x, y\in X}(1_{M_H}\otimes \sigma_{xy}^{bd})L_{cy}L_{ax})
\\
&=&
\sum_{x, y, e, f\in X}
(\mu_r^{A_\sigma}(\sigma_{xy}^{ef})\otimes 1_{A_\sigma}-
1_{A_\sigma}\otimes\mu_l^{A_\sigma}(\sigma_{xy}^{ef}))
(L_{cy}L_{ax}+I_\sigma\otimes L_{fd}L_{eb}+I_\sigma)
\\
&\in&
I_2.
\end{eqnarray*}
This completes the proof.
\end{proof}

We define the $\C$-linear map $\widetilde{\Delta}:
A_\sigma\to A_\sigma\otimes_\C A_\sigma/I_2$
by 
\[
\widetilde{\Delta}(v+I_\sigma)=\bar{\Delta}(v)+I_2
\quad(v\in\C\{ HX\}).
\]
By virtue of Proposition \ref{ar:proposition:DIRI2},
the map $\widetilde{\Delta}$ is well defined.

Since
$\bar{\Delta}(\C\{ HX\}_{\alpha,\beta})
\subset
\sum_{\gamma\in W_X^H}
(A_\sigma)_{\alpha,\gamma}\otimes(A_\sigma)_{\gamma,\beta}$
$(\alpha, \beta\in W_X^H)$,
$\widetilde{\Delta}$
satisfies
$\widetilde{\Delta}(A_\sigma)\subset A_\sigma\tildeotimes A_\sigma$.
Set $\Delta=\widetilde{\Delta}|_{A_\sigma}:
A_\sigma\to A_\sigma\tildeotimes A_\sigma$.
On account of Definition
\ref{hxlag:proposition:hxalghom},
\begin{proposition}
$\Delta$ is a morphism of $\Alg_{(H, X)}$.
\end{proposition}

The next is to construct the counit $\varepsilon:
A_\sigma\to I_{H, X}$
(For $I_{H, X}$, see Proposition \ref{exphxalg:proposition:IHX}).
Let $\varepsilon_{HX}$ denote
the following map from
$HX$ to 
$I_{H, X}$:
\begin{eqnarray*}
&&
\varepsilon_{HX}(\xi)
=
m(\xi)^\star T_{1}
\in(I_{H, X})_{1,1}
=(I_{H, X})_{w_l(\xi),w_r(\xi)}
\quad(\xi\in M_H);
\\
&&
\varepsilon_{HX}(L_{ab})
=
\delta_{ab}T_{[a]}
(=\delta_{ab}T_{[b]})\in
(I_{H, X})_{[a],[b]}
=
(I_{H, X})_{w_l(L_{ab}),w_r(L_{ab})};
\\
&&
\varepsilon_{HX}((L^{-1})_{ab})
=\delta_{ab}T_{[b^{-1}]}
(=\delta_{ab}T_{[a^{-1}]})
\\
&&\qquad\qquad\qquad\quad
\in
(I_{H, X})_{[b^{-1}],[a^{-1}]}
=
(I_{H, X})_{w_l((L^{-1})_{ab}),w_r((L^{-1})_{ab})}.
\end{eqnarray*}
Here, $m: M_H\otimes_\C M_H\to M_H$ is a temporary
notation of the multiplication
defined by $m(f\otimes g)=fg$
$(f, g\in M_H)$;
for $m(\xi)^\star$, see \eqref{exphxalg:eq:fhat}.
This map $\varepsilon_{HX}$
uniquely induces a $\C$-algebra homomorphism
$\bar{\varepsilon}: \C\{ HX\}
\to I_{H, X}$
such that
$\bar{\varepsilon}(v)=\varepsilon_{HX}(v)$
$(v\in HX)$.
\begin{proposition}\label{ar:proposition:eIR0}
$\bar{\varepsilon}(I_\sigma)=\{ 0\}$.
\end{proposition}
\begin{proof}
We shall have established this proposition,
if we prove that $\bar{\varepsilon}=0$
on every generator of the ideal $I_\sigma$.

We will show the proof only for the generator (4)
in Definition \ref{ar:definition:ar}.
Let $g$ be an element of $M_H$.
By the definition of $\bar{\varepsilon}$,
\begin{eqnarray*}
&&\bar{\varepsilon}
(\sum_{x, y\in X}(\sigma_{ac}^{xy}\otimes 1_{M_H})L_{yd}L_{xb}
-\sum_{x, y\in X}(1_{M_H}\otimes \sigma_{xy}^{bd})L_{cy}L_{ax})
(g)(\lambda)
\\
&&
=
\sigma_{ac}^{bd}(\lambda)(g((\lambda d)b)-
g((\lambda c)a))
\quad(\lambda\in H).
\end{eqnarray*}
Because $\sigma$ is a morphism of $\VectH$,
$\sigma_{ac}^{bd}(\lambda)=0$
unless $(\lambda d)b=(\lambda c)a$.
Hence,
\[\bar{\varepsilon}
(\sum_{x, y\in X}(\sigma_{ac}^{xy}\otimes 1_{M_H})L_{yd}L_{xb}
-\sum_{x, y\in X}(1_{M_H}\otimes \sigma_{xy}^{bd})L_{cy}L_{ax})=0,
\]
which is the desired conclusion.
\end{proof}

We denote by $\varepsilon$
the following $\C$-linear map from
$A_\sigma$ to $I_{H, X}$.
\[
\varepsilon(v+I_\sigma)=\bar{\varepsilon}(v)
\quad(v\in\C\{ HX\}).
\]
Proposition \ref{ar:proposition:eIR0}
makes
this definition allowable.
\begin{proposition}
The map $\varepsilon$ is a morphism of $\Alg_{(H, X)}$.
\end{proposition}
It is immediate that
\eqref{hxbialg:eq:catbialg1}
and
\eqref{hxbialg:eq:catbialg2}
hold on
the generators
$f\otimes g+I_\sigma$
$(f, g\in M_H)$,
$L_{ab}+I_\sigma$, and $(L^{-1})_{ab}+I_\sigma$
$(a, b\in W_X^H)$
of $A_\sigma$.
From Definitions \ref{hxbialg:definition:catbialg}
and \ref{hxbialg:definition:hxbialg},
\begin{proposition}\label{ar:proposition:arhxbialg}
$A_\sigma$ is an $(H, X)$-bialgebroid.
\end{proposition}
Let $R(\lambda)$ be
a bijective
dynamical Yang-Baxter map
satisfying the weight zero condition
\eqref{catlop:eq:weightzero}.
If the morphism $\sigma=\sigma_R$ \eqref{catlop:eq:sigmar},
then
the generators $(4)$ of the ideal $I_{\sigma_R}$ are
\begin{equation}
\sum_{x, y\in X}(R_{ac}^{xy}\otimes 1_{M_H})L_{xb}L_{yd}
-
\sum_{x, y\in X}(1_{M_H}\otimes R_{xy}^{bd})L_{cy}L_{ax}
\label{ar:eq:gen4Ir}
\quad
(a, b, c, d\in X).
\end{equation}
Here $R_{ac}^{xy}\in M_H$ is the following function.
\[
R_{ac}^{xy}(\lambda)=\left\{
\begin{array}{ll}
1,&(a, c)=R(\lambda)(x, y),
\\
0,& \mbox{otherwise}.
\end{array}
\right.
\]
We will use the notations $I_R$ and $A_R$
instead of $I_{\sigma_R}$ and $A_{\sigma_R}$, respectively.
From Proposition \ref{ar:proposition:arhxbialg},
\begin{corollary}\label{ar:remark:R}
$A_R$ is an $(H, X)$-bialgebroid.
\end{corollary}

We now recall
the bijective dynamical Yang-Baxter
map $R^{Q_5}(\lambda)$
$(\lambda\in Q_5)$
constructed at the end of Section \ref{section:catlop}.
To shorten notation, we write $R(\lambda)$ for $R^{Q_5}(\lambda)$.
Since $Q_5$ is a quasigroup
(See Definition \ref{catlop:definition:quasigroup}),
$Q_5$ satisfies the assumption
in the beginning of Section
\ref{section:hxalg}; that is,
for any $x\in Q_5$,
the translation map \eqref{hxalg:eq:assumebijec}
$\cdot x: Q_5
\ni\lambda\mapsto\lambda\cdot x\in Q_5$
is bijective.
Before ending this section,
we show
that the
$(Q_5, Q_5)$-bialgebroid
$A_R$
is not a direct sum of the components
$\{ (A_R)_{\alpha,\beta}; \alpha, \beta\in W_{Q_5}^{Q_5}\}$.
This property is a characteristic feature that
distinguish the $(H, X)$-bialgebroid $A_R$ from
the dynamical quantum group
\cite[Section 4.4]{etingof1998}.

We consider the element
$(R_{43}^{12}\otimes 1_{M_{Q_5}})L_{11}L_{22}
+I_R\in A_R$.
By the definition of the counit $\varepsilon: A_R\to I_{H, X}$,
\[
\varepsilon((R_{43}^{12}\otimes 1_{M_{Q_5}})L_{11}L_{22}
+I_R)=(R_{43}^{12})^\star T_{[1]}T_{[2]}.
\]
It follows from \eqref{catlop:eq:RQ501243} that $R_{43}^{12}(0)=1$,
and 
$(R_{43}^{12})^\star T_{[1]}T_{[2]}(1_{M_{Q_5}})(0)=1$
as a result.
Therefore, $(R_{43}^{12}\otimes 1_{M_{Q_5}})L_{11}L_{22}
+I_R\neq 0_{A_R}$.

The relation \eqref{ar:eq:gen4Ir} yields that
\begin{eqnarray}
\nonumber
(R_{43}^{12}\otimes 1_{M_{Q_5}})L_{11}L_{22}
+I_R
&=&
-\sum_{x, y\in Q_5\atop (x, y)\neq (1, 2)}
(R_{43}^{xy}\otimes 1_{M_{Q_5}})L_{x1}L_{y2}
+I_R
\\
&&
+\sum_{x', y'\in Q_5}(1_{M_{Q_5}}\otimes R_{x'y'}^{12})L_{3y'}L_{4x'}
+I_R.
\label{ar:eq:RLL=LLRinAR}
\end{eqnarray}
From the definitions of $w_l$ and $w_r$,
\begin{eqnarray*}
&&
(R_{43}^{xy}\otimes 1_{M_{Q_5}})L_{x1}L_{y2}
+I_R
\in (A_R)_{[xy],[12]}
\quad(x, y\in Q_5),
\\
&&
(1_{M_{Q_5}}\otimes R_{x'y'}^{12})L_{3y'}L_{4x'}
+I_R
\in (A_R)_{[34],[y'x']}
\quad(x', y'\in Q_5).
\end{eqnarray*}

Because of Table \ref{catlop:table:Q5},
the equation $[xy]=[12]$ in $W_{Q_5}^{Q_5}$
$(x, y\in Q_5)$ has a unique solution 
$(x, y)=(1, 2)$.
It follows from
\eqref{ar:eq:RLL=LLRinAR}
that
the $(Q_5, Q_5)$-bialgebroid
$A_R$
is not a direct sum of the components
$\{ (A_R)_{\alpha,\beta}; \alpha, \beta\in W_{Q_5}^{Q_5}\}$.

A crucial point is the weight zero condition
\eqref{catlop:eq:weightzero}
(cf.\ \cite[Section 3.2]{hartwig}),
which does not always imply that
the function
$R_{43}^{12}=0$ unless $[12]=[34]$.

This example
gives birth to the condition $(1)$ in Definition
\ref{hxalg:definition:hxalg}
(cf.\ $\mathfrak h$-algebras \cite[Section 4.1]{etingof1998}
and $\mathfrak h$-prealgebras \cite[Section 2.1]{koelink2007}).
In addition,
this condition $(1)$ suggests that
we employ
morphisms
$a^l_{A, A, A}$, $a^r_{A, A, A}$, $l_A$, and $r_A$
of $\Alg_{(H, X)}$ in
Definition \ref{hxbialg:definition:hxbialg}.
\section{Dynamical representations}
\label{section:dynrep}
A dynamical representation of an $(H, X)$-algebra
$A$ is, by definition,
a pair $\pi_V=(V, \pi_V)$
of $V\in \Ob(\VectH)$
and
$\pi_V: A\to  D_{H, X}(V)\in\Hom(\Alg_{(H, X)})$
(cf.\ \cite[Section 4.2]{etingof1998}).
For $D_{H, X}(V)$, see Proposition
\ref{exphxalg:proposition:DHXV}.

This section deals with two dynamical representations:
one is the trivial representation of an arbitrary
$(H, X)$-bialgebroid
produced by its counit $\varepsilon$;
and the other is
the basic representation
of the $(H, X)$-bialgebroid
$A_\sigma$
given by a Yang-Baxter operator $\sigma$ on $X$
in $\VectH$.
We also explain a category of dynamical representations
before studying a tensor product of the dynamical representations
in the next section.

Let $A$ be an $(H, X)$-bialgebroid
(Definitions \ref{introduction:definition:hxbialg}
and \ref{hxbialg:definition:hxbialg}).
In view of Proposition
\ref{exphxalg:proposition:iota},
the map $\pi_{I_{\VectH}}:=\varphi_0\circ\varepsilon:
A\to D_{H, X}(I_{\VectH})$
is a morphism of $\Alg_{(H, X)}$
(For $I_{\VectH}$, see Section \ref{section:catvecth}).
Hence
\begin{proposition}\label{dynrep:proposition:piIVectH}
The pair 
$(I_{\VectH}, \pi_{I_{\VectH}})$
is a dynamical 
representation of the $(H, X)$-bialgebroid $A$.
\end{proposition}

Assume that the set $X$ is finite,
and let $\sigma$ be a Yang-Baxter operator on $X$
in $\VectH$
(Definition \ref{catlop:definition:YBop}).
We next define
a dynamical representation
$(X, \pi_\sigma)$
of the $(H, X)$-bialgebroid $A_\sigma$
(Proposition \ref{ar:proposition:arhxbialg}).
For $X=(X, \cdot)$, see Proposition
\ref{lop:proposition:X}.
Let $\bar{\pi}_\sigma$
denote
the following map from $HX$ \eqref{ar:eq:HX}
to $D_{H, X}(X)$ that is $\C$-linear on $M_H\otimes_\C M_H$:
\begin{eqnarray*}
&&
\bar{\pi}_\sigma(f\otimes g)=\mu_l^{D_{H, X}(X)}(f)
\mu_r^{D_{H, X}(X)}(g);
\\
&&
\bar{\pi}_\sigma(L_{ab})
=\Gamma_{[a],[b]}^{X}(\bar{\pi}_\sigma(L_{ab})_{[a],[b]}),
\\
&&
\bar{\pi}_\sigma(L_{ab})_{[a],[b]}(\lambda)(x, [b])
=\sum_{y\in X}
([a], y)\sigma(\lambda)_{(a, y)(x, b)};
\\
&&
\bar{\pi}_\sigma((L^{-1})_{ab})
=\Gamma_{[b^{-1}],[a^{-1}]}^{X}
(\bar{\pi}_\sigma((L^{-1})_{ab})_{[b^{-1}],[a^{-1}]}),\\
&&
\bar{\pi}_\sigma((L^{-1})_{ab})_{[b^{-1}],[a^{-1}]}(\lambda)(x, [a^{-1}])
=\sum_{y\in X}
([b^{-1}], y)
(\sigma(\lambda[a^{-1}])^{-1})_{(y, a)(b, x)}. 
\end{eqnarray*}
For $\mu_l^{D_{H, X}(X)}$,
$\mu_r^{D_{H, X}(X)}$,
and
$\sigma(\lambda)_{(a, y)(x, b)}$,
see Proposition \ref{exphxalg:proposition:DHXV}
and
\eqref{intro:equation:matrix}, respectively.
This map $\bar{\pi}_\sigma$ induces a morphism
$\pi_\sigma: A_\sigma\to D_{H, X}(X)$ of $\Alg_{(H, X)}$
(See Remark \ref{isomorph:remark:same}).
As a result,
\begin{proposition}\label{dynrep:proposition:DHXXpiR}
For a Yang-Baxter operator $\sigma$ on $X$
in $\VectH$,
$(X, \pi_\sigma)$
is a dynamical representation of the $(H, X)$-bialgebroid $A_\sigma$.
\end{proposition}
With the aid of Propositions
\ref{lop:proposition:sigmaR}
and \ref{dynrep:proposition:DHXXpiR},
$(X, \pi_{\sigma_R})$
is a dynamical representation of the $(H, X)$-bialgebroid
$A_R$
(Corollary \ref{ar:remark:R}).

To end this section,
we will introduce a category $\DR(A)$
of the dynamical representations of an $(H, X)$-algebra
$A$.
An object of $\DR(A)$ is, by definition, 
a dynamical representation
of $A$.
We denote by $\Ob(\DR(A))$ the class of all objects.

Let $\pi_V=(V, \pi_V)$ and
$\pi_W=(W, \pi_W)$
be objects of $\DR(A)$.
A morphism $f: \pi_V\to \pi_W$
of $\DR(A)$ is a morphism
$f: V\to W$ of the category $\VectH$
such that
\begin{equation}\label{dynrep:eq:defmorphism}
\pi_W(a)\circ m_f=m_f\circ\pi_V(a)
\quad(\forall a\in A).
\end{equation}
Here,
$m_f: \Map(H, \C V)\to \Map(H, \C W)$
is the following $\C$-linear map.
\begin{equation}\label{dynrep:eq:mf}
m_f(g)(\lambda)=f(\lambda)(g(\lambda))
\quad(g\in\Map(H, \C V), \lambda\in H).
\end{equation}
Write $\Hom(\DR(A))$ for the class of 
all morphisms of $\DR(A)$.
For a morphism $f: \pi_V\to\pi_W\in\Hom(\DR(A))$,
we call the objects $\pi_V$ and $\pi_W$ the source $s(f)$
and the target $b(f)$ of the morphism $f$, respectively.

Let $\id_{\pi_V}: \pi_V\to\pi_V$
denote
the morphism $\id_V: V\to V
\in\Hom(\VectH)$.

For morphisms $f$ and $g$ of $\DR(A)$
such that $b(f)=s(g)$,
$m_{g\circ f}=m_g\circ m_f$.
Hence, the morphism $g\circ f$ of the category $\VectH$
is a morphism of $\DR(A)$.
We define the composition $\circ$ of $\DR(A)$
by that of $\VectH$.
\begin{proposition}\label{dynrep:proposition:catDRA}
$\DR(A)$ is a category.
\end{proposition}
\section{Tensor products of dynamical
representations}
\label{section:tpdynrep}
Let $A$ be an $(H, X)$-bialgebroid
(Definitions \ref{introduction:definition:hxbialg}
and \ref{hxbialg:definition:hxbialg}).
In this section, we explain the tensor product $\barotimes:
\DR(A)\times \DR(A)\to \DR(A)$, which makes the category $\DR(A)$
a tensor category
(cf.\ \cite[Section 4.2]{etingof1998}).

Let $V$ and
$W$
be objects of the category
$\VectH$
(See Section \ref{section:catvecth}).
We first define
$\varphi_2(V, W): 
D_{H, X}(V)\tildeotimes D_{H, X}(W)
\to
D_{H, X}(V\barotimes W)\in\Hom(\Alg_{(H, X)})$
(For $D_{H, X}(V)$ and $\tildeotimes$, see Propositions 
\ref{exphxalg:proposition:DHXV}
and \ref{exphxalg:proposition:AtildeotimesB},
respectively).

Let $g$ be an element of $\Map(H, \C (V\barotimes W))$.
By using $g(\lambda)_{(v_1, v_2)}$
\eqref{intro:equation:coeff}, we define the map
$g^{(v)}\in\Map(H, \C W)$ 
$(v\in V)$ by
\[
g^{(v)}(\lambda)=\sum_{v_2\in W}
v_2
g(\lambda)_{(v, v_2)}
\quad(\lambda\in H).
\]

Let $v\in V$.
For simplicity of notation,
we use the same letter $v$
for the constant map from $H$ to
$\C V$
whose value is always $v$;
that is,
$v(\lambda)=v$
$(\lambda\in H)$.

Let $U$ be an element of $\End_\C(\Map(H, \C V))$.
We denote by $U^{(1)}$
the following element
of $\End_\C(\Map(H, \C (V\barotimes W)))$.
\begin{eqnarray*}
U^{(1)}(g)(\lambda)&=&
\sum_{(v_1, v_2)\in V\barotimes W}
(v_1, v_2)
\sum_{v'_1\in V}
U(v'_1)(\lambda v_2)_{v_1}
g(\lambda)_{(v'_1, v_2)}
\\
&&
(g\in\Map(H, \C (V\barotimes W)),
\lambda\in H).
\end{eqnarray*}

Let $U$
be an element of
$D_{H, X}(W)$.
From the definition,
\[
U=\sum_{\alpha,\beta\in W_X^H}
\Gamma_{\alpha,\beta}^W(u_{\alpha\beta})
\quad
(\exists
u_{\alpha,\beta}\in\Hom_{\VectH}
(W\barotimes\{\beta\}, \{\alpha\}\barotimes W)).
\]
We define the element $U^{(2)}$
of $\End_\C(\Map(H, \C (V\barotimes W)))$
by
\begin{eqnarray*}
U^{(2)}(g)(\lambda)&=&
\sum_{(v_1, v_2)\in V\barotimes W}
(v_1, v_2)
\sum_{\alpha,\beta\in W_X^H\atop
v'_2\in W}
u_{\alpha,\beta}(\lambda)_{(\alpha, v_2)(v'_2, \beta)}
g(\lambda\beta)_{(v_1,v'_2)}
\\
&&
(g\in\Map(H, \C (V\barotimes W)),
\lambda\in H).
\end{eqnarray*}
Since
$U^{(2)}(g)(\lambda)=
\sum_{(v_1, v_2)\in V\barotimes W}
(v_1, v_2)
U(g^{(v_1)})(\lambda)_{v_2}$,
the definition of $U^{(2)}$ is unambiguous.

Because the map
\[
D_{H, X}(V)\times D_{H, X}(W)\ni(U_1, U_2)
\mapsto U_1^{(1)}\circ U_2^{(2)}
\in\End_\C(\Map(H, \C (V\barotimes W)))
\]
is $\C$-bilinear,
there exists a unique $\C$-linear map
$\bar{\varphi}_2(V, W)$
from $D_{H, X}(V)\otimes_\C D_{H, X}(W)$
to $\End_\C(\Map(H, \C (V\barotimes W)))$
such that 
\begin{equation}\label{tpdynrep:eq:varphi2}
\bar{\varphi}_2(V, W)(U_1\otimes U_2)=U_1^{(1)}\circ U_2^{(2)}
\quad(U_1\in D_{H, X}(V), U_2\in D_{H, X}(W)).
\end{equation}

We write $I_2$
for
$I_2(D_{H, X}(V), D_{H, X}(W))$
(See
\eqref{exp:eq:I2}).
The map $\bar{\varphi}_2(V, W)$
satisfies
$\bar{\varphi}_2(V, W)(I_2)=\{ 0\}$,
and
it consequently induces
the $\C$-linear map
\[
\varphi_2(V, W): 
D_{H, X}(V)\otimes_\C D_{H, X}(W)/I_2
\to \End_\C(\Map(H, \C (V\barotimes W)))
\]
defined by
$\varphi_2(V, W)(\xi+I_2)=\bar{\varphi}_2(V, W)(\xi)$
$(\xi\in D_{H, X}(V)\otimes_\C D_{H, X}(W))$.

Let $v:
V\barotimes\{\gamma\}\to\{\alpha\}\barotimes V$
and 
$w:
W\barotimes\{\beta\}\to\{\gamma\}\barotimes W$
be morphisms of $\VectH$
$(\alpha, \beta, \gamma\in W_X^H)$.
On account of \eqref{catlop:eq:tensor},
we define the morphism
$v\boxtimes_\gamma w:
(V\barotimes W)\barotimes\{\beta\}\to
\{\alpha\}\barotimes(V\barotimes W)$
of $\VectH$
by
\[
v\boxtimes_\gamma w
=
a_{\{\alpha\}VW}\circ(v\barotimes\id_W)\circ
a_{V\{\gamma\}W}^{-1}\circ
(\id_V\barotimes w)\circ
a_{VW\{\beta\}}.
\]
A simple computation shows
\begin{proposition}\label{tpdynrep:proposition:varphi2}
$\varphi_2(V, W)(\Gamma_{\alpha,\gamma}^V(v)\otimes
\Gamma_{\gamma,\beta}^W(w)+I_2)
=\Gamma_{\alpha,\beta}^{V\barotimes W}(v\boxtimes_\gamma w)$.
\end{proposition}
From Propositions
\ref{exphxalg:proposition:AtildeotimesB}
and
\ref{tpdynrep:proposition:varphi2},
\[\varphi_2(V, W)
((D_{H, X}(V)\tildeotimes D_{H, X}(W))_{\alpha,\beta})
\subset D_{H, X}(V\barotimes W)_{\alpha,\beta},
\]
and consequently
$\varphi_2(V, W)
(D_{H, X}(V)\tildeotimes D_{H, X}(W))
\subset D_{H, X}(V\barotimes W)$.
We continue to write $\varphi_2(V, W)$
for the map
$\varphi_2(V, W)|_{D_{H, X}(V)
\tildeotimes D_{H, X}(W)}$.
\begin{proposition}\label{tpdynrep:proposition:theta}
$\varphi_2(V, W): D_{H, X}(V)
\tildeotimes D_{H, X}(W)\to
D_{H, X}(V\barotimes W)$
is a morphism of $\Alg_{(H, X)}$.
\end{proposition}

The above proposition yields that,
for $\pi_V=(V, \pi_V), \pi_W=(W, \pi_W)\in \Ob(\DR(A))$,
$\pi_V\barotimes \pi_W:=
\varphi_2(V, W)\circ\pi_V\tildeotimes\pi_W
\circ\Delta: A\to D_{H, X}(V\barotimes W)$
is a morphism
of the category $\Alg_{(H, X)}$.
Therefore, $\pi_V\barotimes \pi_W:=
(V\barotimes W, \pi_V\barotimes \pi_W)$
is an object of $\DR(A)$, which is called a tensor product
of the objects $\pi_V$ and $\pi_W$.
\begin{proposition}\label{tpdynrep:prop:tensormorphism}
For $f, g\in\Hom(\DR(A))$,
the tensor product $f\barotimes g$ in $\VectH$ is a morphism
of $\DR(A)$
whose source is
$s(f)\barotimes s(g)$
and whose target is $b(f)\barotimes b(g)$.
\end{proposition}
\begin{proof}
Let $\alpha, \gamma\in W_X^H$,
$f\in\Hom_{\VectH}(V, V')$,
$u\in\Hom_{\VectH}(V\barotimes\{\gamma\}, \{\alpha\}\barotimes V)$,
and
$u'\in\Hom_{\VectH}(V'\barotimes\{\gamma\}, \{\alpha\}\barotimes V')$.
\begin{lemma}\label{tpdynrep:lem:equivcommute}
$\Gamma_{\alpha,\gamma}^{V'}(u')\circ m_f
=m_f\circ\Gamma_{\alpha,\gamma}^{V}(u)$,
if and only if
$u'\circ(f\barotimes\id_{\{\gamma\}})
=(\id_{\{\alpha\}}\barotimes f)\circ u$.
\end{lemma}
Let $\beta\in W_X^H$,
$g\in\Hom_{\VectH}(W, W')$,
$w\in\Hom_{\VectH}(W\barotimes\{\beta\}, \{\gamma\}\barotimes W)$,
and
$w'\in\Hom_{\VectH}(W'\barotimes\{\beta\}, \{\gamma\}\barotimes W')$.
The above lemma and Proposition \ref{tpdynrep:proposition:varphi2}
imply that
\begin{lemma}
If
$\Gamma_{\alpha,\gamma}^{V'}(u')\circ m_f
=m_f\circ\Gamma_{\alpha,\gamma}^{V}(u)$
and
$\Gamma_{\gamma,\beta}^{W'}(w')\circ m_g
=m_g\circ\Gamma_{\gamma,\beta}^{W}(w)$,
then
\begin{eqnarray*}
&&
\varphi_2(V', W')(\Gamma_{\alpha,\gamma}^{V'}(u')\otimes 
\Gamma_{\gamma,\beta}^{W'}(w')+I_2)\circ
m_{f\barotimes g}
\\
&=&
m_{f\barotimes g}
\circ
\varphi_2(V, W)(\Gamma_{\alpha,\gamma}^{V}(u)\otimes 
\Gamma_{\gamma,\beta}^{W}(w)+I_2).
\end{eqnarray*}
\end{lemma}
This lemma immediately induces Proposition
\ref{tpdynrep:prop:tensormorphism}.
\end{proof}

For abbreviation,
we continue to write $f\barotimes g$
for the tensor product of the morphisms $f$ and $g$ of $\DR(A)$. 
\begin{proposition}\label{tpdynrep:proposition:barotimes}
$\barotimes: \DR(A)\times \DR(A)\to \DR(A)$ is a functor.
\end{proposition}

Let $I_{\DR(A)}$ denote the object
$(I_{\VectH}, \pi_{I_{\VectH}})\in \Ob(\DR(A))$
(See Proposition \ref{dynrep:proposition:piIVectH}).
The final task of this section is to explain
the associativity constraint
$a: \barotimes(\barotimes\times\id)
\to\barotimes(\id\times\barotimes)$,
the left unit constraint $l: \barotimes(I_{\DR(A)}\times
\id)\to\id$,
and the right unit constraint
$r: \barotimes(\id\times I_{\DR(A)})\to\id$.
Let 
$\pi_i=(V_i, \pi_i)$
$(i=1, 2, 3)$
be objects of $\DR(A)$.
We denote by $a_{\pi_1,\pi_2,\pi_3}$,
$l_{\pi_1}$, and $r_{\pi_1}$
the following morphisms of $\VectH$:
\begin{eqnarray*}
&&
a_{\pi_1,\pi_2,\pi_3}=a_{V_1,V_2,V_3}:
(V_1\barotimes V_2)\barotimes V_3
\to 
V_1\barotimes (V_2\barotimes V_3);
\\
&&
l_{\pi_1}=l_{V_1}: I_{\VectH}\barotimes V_1\to V_1;
\ 
r_{\pi_1}=r_{V_1}: V_1\barotimes I_{\VectH}\to V_1.
\end{eqnarray*}
\begin{theorem}\label{tydynrep:theorem:tensorcategory}
If $A$ is an $(H, X)$-bialgebroid
$($Definition $\ref{hxbialg:definition:hxbialg})$,
then $\DR(A)$, together with
$\barotimes$, $I_{\DR(A)}$, $a$, $l$, and $r$,
is a tensor category.
\end{theorem}

Because of Proposition \ref{ar:proposition:arhxbialg}
and Corollary \ref{ar:remark:R},
\begin{corollary}\label{tydynrep:cor:tensorcategory}
$\DR(A_\sigma)$ is a tensor category,
and so is $DR(A_R)$.
\end{corollary}

We will prove Theorem \ref{tydynrep:theorem:tensorcategory}
in the next section.
\section{Proof of Theorem
\ref{tydynrep:theorem:tensorcategory}}
\label{section:prtencat}
Before starting the proof of
Theorem
\ref{tydynrep:theorem:tensorcategory},
let us introduce the category $C_{is}$ called
the groupoid of isomorphisms
of the category $C$
\cite[XI.1.1 Example 2]{kassel}.
The notation $C_{is}$ means
the (broad) subcategory of $C$
whose objects are those of $C$
and whose morphisms are isomorphisms of $C$.
If $C$ is a tensor category, then $C_{is}$ is a tensor subcategory
of $C$;
in other words,
$C_{is}$ is a tensor category with respect to
$\otimes$, $I$, $a$, $l$, and $r$
of the tensor category $C$.

We will first explain a functor
$D_{H, X}: {\VectH}_{is}\to\Alg_{(H, X)}$,
which
plays an essential role in the proof of
Theorem
\ref{tydynrep:theorem:tensorcategory}
(For $\VectH$ and $\Alg_{(H, X)}$, see Sections \ref{section:catvecth}
and \ref{section:hxalg}).

For an object $V$ of the category ${\VectH}_{is}$,
we have already defined the object $D_{H, X}(V)$
of $\Alg_{(H, X)}$
in Proposition \ref{exphxalg:proposition:DHXV}.
Let $f: V\to W$ be a morphism of the category ${\VectH}_{is}$.
Define a $\C$-linear map $D_{H, X}(f)$ from $\End_\C(\Map(H, \C V))$
to $\End_\C(\Map(H, \C W))$
by
\[
D_{H, X}(f)(U)=m_f\circ U\circ (m_f)^{-1}
\quad(U\in \End_\C(\Map(H, \C V))).
\]
For $m_f$, see \eqref{dynrep:eq:mf}.
We note that
$(m_f)^{-1}=m_{f^{-1}}$.
Because $D_{H, X}(f)$ satisfies that
$D_{H, X}(f)(D_{H, X}(V)_{\alpha,\beta})\subset D_{H, X}(W)_{\alpha,\beta}$
$(\forall\alpha, \beta\in W_X^H)$,
we use the same symbol
$D_{H, X}(f)$ for
the restriction
$D_{H, X}(f)|_{D_{H, X}(V)}$.
By Definition \ref{hxlag:proposition:hxalghom},
\begin{proposition}\label{prtencat:proposition:DHXfunc}
$D_{H, X}(f): D_{H, X}(V)\to D_{H, X}(W)\in\Hom(\Alg_{(H, X)})$.
Furthermore,
$D_{H, X}: {\VectH}_{is}\to \Alg_{(H, X)}$ is a functor.
\end{proposition}

The morphisms $\varphi_2(V, W)$
in Proposition \ref{tpdynrep:proposition:theta}
$(V, W\in \Ob(\VectH))$
are relevant to this functor $D_{H, X}$.
To be more precise,
\begin{proposition}
$\varphi_2: \tildeotimes(D_{H, X}\times D_{H, X})\to
D_{H, X}\barotimes$ is a natural transformation.
\end{proposition}

Let $V_1, V_2$, and $V_3$ be objects of the category ${\VectH}_{is}$.
The next task is to construct two families of morphisms
of the category $\Alg_{(H, X)}$:
\begin{eqnarray*}
&&
\varphi^l_3(V_1, V_2, V_3): 
D_{H, X}(V_1)\tildeotimes D_{H, X}(V_2)\tildeotimes
D_{H, X}(V_3)\to
D_{H, X}((V_1\barotimes V_2)\barotimes V_3);
\\
&&
\varphi^r_3(V_1, V_2, V_3): 
D_{H, X}(V_1)\tildeotimes D_{H, X}(V_2)
\tildeotimes D_{H, X}(V_3)\to
D_{H, X}(V_1\barotimes (V_2\barotimes V_3)).
\end{eqnarray*}

We define the map
\[
\doublebarTheta:
(D_{H, X}(V_1)\otimes_\C D_{H, X}(V_2))
\times D_{H, X}(V_3)
\to
\End(\Map(H, \C (V_1\barotimes V_2)\barotimes V_3))
\]
as follows
(For $\bar{\varphi}_2(V_1, V_2)$, see
\eqref{tpdynrep:eq:varphi2}).
\[
\doublebarTheta(w, x)
=
\bar{\varphi}_2(V_1, V_2)(w)^{(1)}\circ x^{(2)}
\ \ 
(w\in D_{H, X}(V_1)\otimes_\C D_{H, X}(V_2),
x\in D_{H, X}(V_3)).
\]
Since this map is $\C$-bilinear,
there exists a unique $\C$-linear map
\[
\bar{\varphi}^l_3: 
(D_{H, X}(V_1)\otimes_\C D_{H, X}(V_2))
\otimes_\C D_{H, X}(V_3)
\to
\End(\Map(H, \C (V_1\barotimes V_2)\barotimes V_3))
\]
such that
$\bar{\varphi}^l_3(w\otimes x)=\doublebarTheta(w, x)$.

We set $I_3^l=I_3^l(D_{H, X}(V_1), D_{H, X}(V_2), D_{H, X}(V_3))$
(See \eqref{exphxalg:eq:ABCl}).
The map $\bar{\varphi}^l_3$ satisfies
$\bar{\varphi}^l_3(I_3^l)=\{ 0\}$,
and it hence induces the following $\C$-linear map
$\varphi^l_3$.
\[
\varphi^l_3:
(D_{H, X}(V_1)\otimes_\C D_{H, X}(V_2))
\otimes_\C D_{H, X}(V_3)/I_3^l
\to
\End(\Map(H, \C (V_1\barotimes V_2)\barotimes V_3)).
\]
For any $\alpha, \beta\in W_X^H$,
\[
\varphi^l_3((D_{H, X}(V_1)\tildeotimes D_{H, X}(V_2)\tildeotimes 
D_{H, X}(V_3))_{\alpha,\beta})
\subset
D_{H, X}((V_1\barotimes V_2)\barotimes V_3)_{\alpha,\beta};
\]
and we will use the symbol $\varphi^l_3(V_1, V_2, V_3)$
for the restriction of the map $\varphi^l_3$.
\begin{eqnarray*}
&&
\varphi^l_3(V_1, V_2, V_3)
=\varphi^l_3|_{D_{H, X}(V_1)\tildeotimes D_{H, X}(V_2)\tildeotimes 
D_{H, X}(V_3)}:
\\
&&
D_{H, X}(V_1)\tildeotimes D_{H, X}(V_2)\tildeotimes 
D_{H, X}(V_3)
\to
D_{H, X}((V_1\barotimes V_2)\barotimes V_3).
\end{eqnarray*}
\begin{proposition}\label{prtencat:proposition:Theta1}
The map $\varphi^l_3(V_1, V_2, V_3)$ is
a morphism of $\Alg_{(H, X)}$.
\end{proposition}

We apply this argument again
to obtain
the morphism $\varphi^r_3(V_1, V_2, V_3)$
of the category $\Alg_{(H, X)}$.

The definition of pre-tensor functors is similar to that of the tensor functors
\cite[Definition XI.4.1]{kassel}.
\begin{definition}\label{prtencat:definition:pretensorfunc}
Let $C=(C, \barotimes, I_C, a, l^C, r^C)$ be a tensor category,
and let 
$D=(D, \tildeotimes, -\tildeotimes -\tildeotimes -, I_D, a^l, a^r, l^D, r^D)$
be a pre-tensor category
$($Definition $\ref{hxbialg:definition:pretensorcat}$$)$.
A pre-tensor functor from $C$ to $D$ is a quintet $(F, \varphi_0,
\varphi_2, \varphi_3^l, \varphi_3^r)$ where
$F: C\to D$ is a functor,
$\varphi_0$ is a morphism from $I_D$ to $F(I_C)$,
$\varphi_2: \tildeotimes(F\times F)\to F\barotimes$ is
a natural transformation,
and
\begin{eqnarray*}
&&
\varphi_3^l(U, V, W): F(U)\tildeotimes F(V)\tildeotimes F(W)\to
F((U\barotimes V)\barotimes W),
\\
&&
\varphi_3^r(U, V, W): F(U)\tildeotimes F(V)\tildeotimes F(W)\to
F(U\barotimes (V\barotimes W))
\end{eqnarray*}
are
families of morphisms
$(U, V, W\in \Ob(C))$
such that
\begin{eqnarray}
&&\label{prtencat:eq:pretenfunc1}
F(a_{U,V,W})\circ\varphi^l_3(U, V, W)
=
\varphi^r_3(U, V, W),
\\&&\nonumber
\varphi^l_3(U, V, W)\circ a^l_{F(U), F(V), F(W)}
\\\label{prtencat:eq:pretenfunc2}
&&\qquad\qquad
=
\varphi_2(U\barotimes V, W)\circ
\varphi_2(U, V)\tildeotimes\id_{F(W)},
\\&&\nonumber
\varphi^r_3(U, V, W)\circ a^r_{F(U), F(V), F(W)}
\\\label{prtencat:eq:pretenfunc3}&&\qquad\qquad
=
\varphi_2(U, V\barotimes W)\circ
\id_{F(U)}\tildeotimes\varphi_2(V, W),
\\&&\label{prtencat:eq:pretenfunc4}
l^D_{F(U)}=
F(l^C_U)\circ\varphi_2(I_C, U)\circ
\varphi_0\tildeotimes\id_{F(U)},
\\&&\label{prtencat:eq:pretenfunc5}
r^D_{F(U)}=
F(r^C_U)\circ\varphi_2(U, I_C)\circ
\id_{F(U)}\tildeotimes \varphi_0
\end{eqnarray}
for all objects $U, V, W$ of $C$.
\end{definition}
\begin{proposition}\label{prtencat:proposition:property}
$D_{H, X}=(D_{H, X}, \varphi_0, \varphi_2, \varphi_3^l, \varphi_3^r)$
is a pre-tensor functor from ${\VectH}_{is}$
to $\Alg_{(H, X)}$.
\end{proposition}
\begin{proof}
On account of Propositions \ref{exphxalg:proposition:iota}
and \ref{prtencat:proposition:DHXfunc}--\ref{prtencat:proposition:Theta1}
with the family of the morphisms
$\varphi^r_3(V_1, V_2, V_3)$
of $\Alg_{(H, X)}$,
it is sufficient to show
\eqref{prtencat:eq:pretenfunc1}--\eqref{prtencat:eq:pretenfunc5}.
The detailed verification is left to the reader.
\end{proof}

We now prove
Theorem
\ref{tydynrep:theorem:tensorcategory}.
The proof is based on the concepts in category theory:
the pre-tensor category, the categorical bialgebroid,
and the pre-tensor functor.
Let $C$ be a tensor category,
$F$ a pre-tensor functor
from $C_{is}$
to a pre-tensor category
$D$,
and
$A$ a categorical bialgebroid of the category $D$.
We follow the notations of
Definitions \ref{hxbialg:definition:catbialg}
and
\ref{prtencat:definition:pretensorfunc}.

For the proof, we first introduce a tensor category
$\PDR_F(A)$,
of which $\DR(A)$ will be a tensor subcategory,
if $F=D_{H, X}: {\VectH}_{is}\to\Alg_{(H, X)}$.

An object of 
$\PDR_F(A)$
is, by definition,
a pair $(V, \pi_V)$
of $V\in \Ob(C)(=\Ob(C_{is}))$
and $\pi_V: A\to F(V)\in\Hom(D)$.

Let $\pi_V=(V, \pi_V)$ and
$\pi_W=(W, \pi_W)$
be objects of $\PDR_F(A)$.
A morphism $f: \pi_V\to \pi_W$
of $\PDR_F(A)$ is a morphism
$f: V\to W$ of the category $C$.
The objects $\pi_V$ and $\pi_W$ are called the source $s(f)$
and the target $b(f)$ of the morphism $f$, respectively.

Let $\id_{\pi_V}: \pi_V\to\pi_V$ denote
the morphism $\id_V: V\to V\in\Hom(C)$.

We define the composition $\circ$ of $\PDR_F(A)$
by that of $C$.
\begin{proposition}
$\PDR_F(A)$ is a category.
\end{proposition}
Note that we have actually proved that 
$\PDR_F(A)$ is a category, if
$A$ is an object of $D$.

The next task is to define a tensor product
$\barotimes: \PDR_F(A)\times \PDR_F(A)\to \PDR_F(A)$.
Let $(V, \pi_V)$
and $(W, \pi_W)$ be objects of $\PDR_F(A)$.
Set 
\[(V, \pi_V)\barotimes(W, \pi_W)=
(V\barotimes W, \varphi_2(V, W)\circ\pi_V\tildeotimes\pi_W\circ\Delta).
\]
For $f, g\in\Hom(\PDR_F(A))$,
we define $f\barotimes g: s(f)\barotimes s(g)\to b(f)\barotimes b(g)$
by $f\barotimes g\in\Hom(C)$.
\begin{proposition}
$\barotimes: \PDR_F(A)\times \PDR_F(A)\to \PDR_F(A)$
is a functor.
\end{proposition}

We will explain
the unit $I_{\PDR_F(A)}$,
the associativity constraint
$a: \barotimes(\barotimes\times\id)
\to\barotimes(\id\times\barotimes)$,
the left unit constraint $l: \barotimes(I_{\PDR_F(A)}\times
\id)\to\id$,
and the right unit constraint
$r: \barotimes(\id\times I_{\PDR_F(A)})\to\id$.
Let $I_{\PDR_F(A)}$ denote the object
$(I_C, \varphi_0\circ\varepsilon)\in \Ob(\PDR_F(A))$.
Let 
$\pi_i=(V_i, \pi_i)$,
$(i=1, 2, 3)$
be objects of $\PDR_F(A)$.
We denote by $a_{\pi_1,\pi_2,\pi_3}$,
$l_{\pi_1}$, and $r_{\pi_1}$
the following morphisms of $C$:
\begin{eqnarray*}
&&
a_{\pi_1,\pi_2,\pi_3}=a_{V_1,V_2,V_3}:
(V_1\barotimes V_2)\barotimes V_3
\to 
V_1\barotimes (V_2\barotimes V_3);
\\
&&
l_{\pi_1}=l_{V_1}^C: I_{C}\barotimes V_1\to V_1;
\ 
r_{\pi_1}=r_{V_1}^C: V_1\barotimes I_{C}\to V_1.
\end{eqnarray*}
\begin{proposition}
$\PDR_F(A)$, together with
$\barotimes$, $I_{\PDR_F(A)}$, $a$, $l$, and $r$,
is a tensor category.
\end{proposition}

Let $\DR_F(A)$ be a subcategory of $\PDR_F(A)$
such that:
\begin{enumerate}
\item
if $\pi_V, \pi_W\in \Ob(\DR_F(A))$,
then
$\pi_V\barotimes\pi_W\in \Ob(\DR_F(A))$;
\item
if $f, g\in\Hom(\DR_F(A))$,
then
$f\barotimes g\in\Hom_{\DR_F(A)}
(s(f)\barotimes s(g), b(f)\barotimes b(g))$;
\item
$I_{\PDR_F(A)}\in \Ob(\DR_F(A))$;
\item
if $f: V\to W\in\Hom(C_{is})$
satisfies $F(f)\circ\pi_V=\pi_W$,
then
$f$ is a morphism of $\DR_F(A)$
whose source is $\pi_V$ and whose target is $\pi_W$.
\end{enumerate}
This $\DR_F(A)$ is a generalization of the category $\DR(A)$.
\begin{proposition}\label{prtencat:proposition:DRFA}
$\DR_F(A)$ is a tensor subcategory of $\PDR_F(A)$.
\end{proposition}
\begin{proof}
Let 
$\pi_i=(V_i, \pi_i)$
$(i=1, 2, 3)$
be objects of $\DR_F(A)$.
We will only show that
$a_{\pi_1,\pi_2,\pi_3},
l_{\pi_1}
\in\Hom(\DR_F(A))$.

We note that $a_{\pi_1,\pi_2,\pi_3}(=a_{V_1, V_2, V_3})$
and
$l_{\pi_1}(=l_{V_1}^C)$
are morphisms of $C_{is}$.
On account of (4) that the subcategory $\DR_F(A)$ satisfies,
it suffices to prove that:
\begin{eqnarray}
&&\label{prtencat:eq:DHXa}
F(a_{V_1, V_2, V_3})\circ(\pi_1\barotimes\pi_2)
\barotimes\pi_3
=\pi_1\barotimes(\pi_2\barotimes\pi_3);
\\&&\nonumber
F(l_{V_1}^C)\circ\pi_{I_{\DR_F(A)}}\barotimes\pi_1
=\pi_1.
\end{eqnarray}

Figures \ref{prtencat:figure:1}
and
\ref{prtencat:figure:2} yield these formulas.
For example, let us consider Figure \ref{prtencat:figure:1}.
By Definitions \ref{hxbialg:definition:pretensorcat},
\ref{hxbialg:definition:catbialg},
and
\ref{prtencat:definition:pretensorfunc},
every small diagram commutes;
so does the outside one.
This is equivalent to \eqref{prtencat:eq:DHXa}.
\end{proof}
The proof of the above proposition strongly depends on
the assumption (4) of the subcategory $\DR_F(A)$.

Let $A$ be an $(H, X)$-bialgebroid
(Definition \ref{hxbialg:definition:hxbialg}).
We set 
$F=D_{H, X}: {\VectH}_{is}\to\Alg_{(H, X)}$;
indeed,
$C=\VectH$
and $D=\Alg_{(H, X)}$.
From
Propositions
\ref{dynrep:proposition:piIVectH},
\ref{dynrep:proposition:catDRA},
\ref{tpdynrep:proposition:barotimes},
and
\eqref{dynrep:eq:defmorphism},
$\DR(A)$ is a subcategory of $\PDR_{D_{H, X}}(A)$
satisfying
(1)--(4).
By virtue of Proposition \ref{prtencat:proposition:DRFA},
$\DR(A)$ is a tensor subcategory
of $\PDR_{D_{H, X}}(A)$.
We have thus proved Theorem \ref{tydynrep:theorem:tensorcategory}.
\begin{figure}
\bigskip
\centering
\setlength{\unitlength}{.75cm}
\begin{picture}(12,13.5)\thicklines
\put(0.9,13.2){$A\tildeotimes A$}
\put(6.1,13.2){$A$}
\put(10.9,13.2){$A\tildeotimes A$}
\put(5.7,13.3){\vector(-1,0){3.5}}
\put(7,13.3){\vector(1,0){3.5}}
\put(4,13.5){$\Delta$}
\put(8.5,13.5){$\Delta$}
\put(1.4,13){\vector(0,-1){2}}
\put(11.4,13){\vector(0,-1){2}}
\put(1.6,12){$\Delta\tildeotimes\id_A$}
\put(9.8,12){$\id_A\tildeotimes\Delta$}
\put(.5,10.5){$(A\tildeotimes A)\tildeotimes A$}
\put(5.5,10.5){$A\tildeotimes A\tildeotimes A$}
\put(10.5,10.5){$A\tildeotimes(A\tildeotimes A)$}
\put(2.9,10.6){\vector(1,0){2.3}}
\put(10.3,10.6){\vector(-1,0){2.7}}
\put(3.4,10.9){$a^l_{A, A, A}$}
\put(8.3,10.9){$a^r_{A, A, A}$}
\put(1.4,10.3){\vector(0,-1){2}}
\put(6.3,10.3){\vector(0,-1){5}}
\put(11.4,10.3){\vector(0,-1){2}}
\put(1.6,9.2){$(\pi_1\tildeotimes\pi_2)\tildeotimes\pi_3$}
\put(8.8,9.2){$\pi_1\tildeotimes(\pi_2\tildeotimes\pi_3)$}
\put(6.5,8.7){$\pi_1\tildeotimes \pi_2\tildeotimes \pi_3$}
\put(0,7.7){$(F(V_1)\tildeotimes F(V_2))
\tildeotimes F(V_3)$}
\put(8.5,7.7){$F(V_1)\tildeotimes (F(V_2)
\tildeotimes F(V_3))$}
\put(2.5,7.4){\vector(1,-1){2.1}}
\put(10.8,7.4){\vector(-1,-1){2.1}}
\put(1.4,7.4){\vector(0,-1){5}}
\put(11.4,7.4){\vector(0,-1){5}}
\put(4,6.2){$a^l_{V_1,V_2,V_3}$}
\put(7.5,6.2){$a^r_{V_1,V_2,V_3}$}
\put(1.6,4.1){$\varphi_2(V_1\barotimes V_2,V_3)\circ$}
\put(1.55,3.55){$\circ\varphi_2(V_1,V_2)
\tildeotimes\id_3$}
\put(8,4.1){$\varphi_2(V_1,V_2\barotimes V_3)\circ$}
\put(8,3.55){$\circ\id_1
\tildeotimes\varphi_2(V_2,V_3)$}
\put(4.7,4.7){$F(V_1)\tildeotimes F(V_2)
\tildeotimes F(V_3)$}
\put(6,4.3){\vector(-4,-3){2.7}}
\put(6.8,4.3){\vector(4,-3){2.7}}
\put(4.55,2.9){$\varphi^l_3(V_1,V_2,V_3)$}
\put(6.2,2.33){$\varphi^r_3(V_1,V_2,V_3)$}
\put(0,1.7){$F((V_1\barotimes V_2)\barotimes V_3)$}
\put(9.7,1.7){$F(V_1\barotimes(V_2\barotimes V_3))$}
\put(3.6,1.75){\vector(1,0){5.6}}
\put(5,1.3){$F(a_{V_1,V_2,V_3})$}
\put(2,.5){$a^l_{V_1,V_2,V_3}=a^l_{F(V_1),
F(V_2),F(V_3)}$, $\id_1=\id_{F(V_1)}$,}
\put(2,-.2){$a^r_{V_1,V_2,V_3}=a^r_{F(V_1),
F(V_2),F(V_3)}$, $\id_3=\id_{F(V_3)}$.}
\end{picture}
\smallskip
\caption{$F(a_{V_1, V_2, V_3})\circ(\pi_1\barotimes\pi_2)
\barotimes\pi_3
=\pi_1\barotimes(\pi_2\barotimes\pi_3)$
\label{prtencat:figure:1}}
\end{figure}
\begin{figure}
\bigskip
\centering
\setlength{\unitlength}{.75cm}
\begin{picture}(12,9)\thicklines
\put(10.4,8.3){$A$}
\put(6.15,8.3){$A\tildeotimes A$}
\put(10.3,8.4){\vector(-1,0){3}}
\put(8.6,8.6){$\Delta$}
\put(10.6,8.1){\vector(0,-1){2}}
\put(6.6,8.1){\vector(0,-1){2}}
\put(6,8.2){\vector(-2,-1){4}}
\put(9.8,7.1){$\id_A$}
\put(6.8,7.1){$\varepsilon\tildeotimes\id_A$}
\put(1,7.1){$\pi_{I_{\VectH}}\tildeotimes\id_A$}
\put(10.4,5.6){$A$}
\put(5.9,5.6){$I_D\tildeotimes A$}
\put(0,5.6){$F(I_C)\tildeotimes A$}
\put(7.4,5.7){\vector(1,0){2.7}}
\put(5.5,5.7){\vector(-1,0){3.3}}
\put(8.6,5.9){$l_A^D$}
\put(3.8,5.9){$\varphi_0\tildeotimes\id_A$}
\put(10.6,5.4){\vector(0,-1){4.6}}
\put(6.6,5.4){\vector(0,-1){2}}
\put(6.8,4.3){$\id_{I_D}\tildeotimes\pi_1$}
\put(1.1,5.4){\vector(0,-1){2}}
\put(1.3,4.3){$\id_{F(I_C)}\tildeotimes\pi_1$}
\put(10,2.8){$\pi_1$}
\put(0,2.8){$F(I_C)\tildeotimes F(V_1)$}
\put(5.5,2.9){\vector(-1,0){2.5}}
\put(5.7,2.8){$I_D\tildeotimes F(V_1)$}
\put(6.6,2.6){\vector(3,-2){3}}
\put(8.3,1.6){$l_{F(V_1)}^D$}
\put(1.1,2.6){\vector(0,-1){2}}
\put(3.6,2.2){$\varphi_0\tildeotimes \id_{F(V_1)}$}
\put(1.3,1.6){$\varphi_2(I_C,V_1)$}
\put(9.8,0){$F(V_1)$}
\put(0,0){$F(I_C\barotimes V_1)$}
\put(2.6,0.1){\vector(1,0){6.7}}
\put(5.4,.4){$F(l_{V_1}^C)$}
\end{picture}
\caption{$F(l_{V_1}^C)\circ\pi_{I_{DR_F(A)}}\barotimes\pi_1
=\pi_1$
\label{prtencat:figure:2}}
\end{figure}
\section{Tensor categories $\Rep R$ and $\DR(A_R)$}
\label{section:isom}
Let $\sigma: X\barotimes X\to X\barotimes X$ be a morphism of
$\VectH$.
In this section, we show that the tensor categories
$\Rep \sigma$ in Proposition \ref{lop:theorem} and
$\DR(A_\sigma)$ in Corollary \ref{tydynrep:cor:tensorcategory}
are isomorphic;
in fact,
we will construct tensor functors $F: \DR(A_\sigma)\to\Rep \sigma$
and $G: \Rep \sigma\to \DR(A_\sigma)$ such that
$FG=\id_{\Rep \sigma}$
and $GF=\id_{\DR(A_\sigma)}$
as tensor functors.
In particular,
$\Rep R$ 
is isomorphic to $DR(A_R)$
(See Proposition \ref{catlop:prop:repR}
and Corollary \ref{tydynrep:cor:tensorcategory}).
As was mentioned in Section \ref{section:ar},
the set $X$ is required to be finite
in this section.

The first task is to introduce a (strict) tensor functor $F$
\cite[Definition XI.4.1]{kassel}.

Let $\pi=(V, \pi)$ be an object of $\DR(A_\sigma)$.
For $a, b\in X$,
we define the morphism $\pi(L_{ab}+I_\sigma)_{[a],[b]}:
V\barotimes\{[b]\}\to\{[a]\}\barotimes V$
of $\VectH$ by
$\pi(L_{ab}+I_\sigma)=
\Gamma_{[a],[b]}^V(\pi(L_{ab}+I_\sigma)_{[a],[b]})$.
Because of Proposition
\ref{hxalg:proposition:unique}
and the fact that
$\pi(L_{ab}+I_\sigma)\in D_{H, X}(V)_{[a],[b]}$,
this definition makes sense.

Let
$L_{F(\pi)}:
V\barotimes X\to
X\barotimes V$
denote the following morphism of $\VectH$
(See 
Proposition
\ref{vecth:proposition:morphismsum}
and \eqref{vecth:eq:injproj}).
\begin{equation}
\label{isom:eq:lfpi}
L_{F(\pi)}
=
\sum_{a, b\in X}
(i_X^{\{a\}}\circ
\iota_{\{ a\}}^{\{ [a]\}})\barotimes\id_V
\circ
\pi(L_{ab}+I_\sigma)_{[a],[b]}
\circ
\id_V\barotimes (\iota^{\{ b\}}_{\{ [b]\}}
\circ p_{\{ b\}}^{X}).
\end{equation}
For the isomorphism
$\iota_{\{ a\}}^{\{ [a]\}}: \{[a]\}\to\{ a\}$ of $\VectH$,
see Propositions
\ref{vecth:proposition:subobject},
\ref{lop:proposition:X},
\ref{hxalg:proposition:wxh},
and
\eqref{catvecth:equation:onepoint}.
\begin{proposition}
$L_{F(\pi)}$ is an isomorphism.
\end{proposition}
\begin{proof}
Let 
$u\in \Hom_{\VectH}(V\barotimes\{\alpha\}, \{\beta\}\barotimes V)$
$(\alpha, \beta\in W_X^H)$.
We will denote by $u^\vee: \{\beta^{-1}\}\barotimes V\to V\barotimes
\{\alpha^{-1}\}$
the following morphism of $\VectH$.
\begin{eqnarray}\nonumber
u^\vee
&=&
l_{V\barotimes\{\alpha^{-1}\}}\circ
\iota_{I_{\VectH}}^{\{\beta^{-1}\}\barotimes\{\beta\}}
\barotimes\id_{V\barotimes\{\alpha^{-1}\}}
\circ
a^{-1}_{\{\beta^{-1}\}\{\beta\}V\barotimes\{\alpha^{-1}\}}
\circ
\\\nonumber
&&
\circ
\id_{\{\beta^{-1}\}}\barotimes
(a_{\{\beta\}V\{\alpha^{-1}\}}\circ
u\barotimes\id_{\{\alpha^{-1}\}}\circ
a^{-1}_{V\{\alpha\}\{\alpha^{-1}\}}
\circ
\id_V\barotimes
\iota_{\{\alpha\}\barotimes\{\alpha^{-1}\}}^{I_{\VectH}})
\circ
\\\label{isom:eq:vee}
&&
\circ
a_{\{\beta^{-1}\}V I_{\VectH}}\circ
r_{\{\beta^{-1}\}\barotimes V}^{-1}.
\end{eqnarray}

We set $L_{F(\pi)}^{-1}: X\barotimes V\to V\barotimes X$
for
\begin{eqnarray}\nonumber
L_{F(\pi)}^{-1}&=&
\sum_{a, b\in X}
\id_V\barotimes(i_X^{\{a\}}\circ
\iota_{\{ a\}}^{\{ [a]\}})
\circ
(\pi((L^{-1})_{ab}+I_\sigma)_{[b^{-1}],[a^{-1}]})^{\vee}
\circ
\\\label{isom:eq:lfpi-1}
&&
\qquad\circ
(\iota^{\{ b\}}_{\{ [b]\}}
\circ p_{\{ b\}}^{X})\barotimes\id_V.
\end{eqnarray}
Here, $\pi((L^{-1})_{ab}+I_\sigma)_{[b^{-1}],[a^{-1}]}:
V\barotimes\{[a^{-1}]\}
\to\{[b^{-1}]\}\barotimes V$ is 
the morphism of $\VectH$ defined by
$\pi((L^{-1})_{ab}+I_\sigma)=\Gamma_{[b^{-1}],[a^{-1}]}^V
(\pi((L^{-1})_{ab}+I_\sigma)_{[b^{-1}],[a^{-1}]})$
(See Proposition \ref{hxalg:proposition:unique}).
\begin{lemma}
\label{isom:lemma:veeandcircV}
For morphisms
$u: V\barotimes\{[c]\}\to\{[a]\}\barotimes V$
and
$v: V\barotimes\{[c^{-1}]\}\to\{[b^{-1}]\}\barotimes V$
of $\VectH$
$(a, b, c\in X)$,
\begin{eqnarray*}
u\circ v^\vee
&=&
l_{\{[a]\}\barotimes V}
\circ
\iota_{I_{\VectH}}^{\{[b]\}\barotimes\{[b^{-1}]\}}
\barotimes
\id_{\{[a]\}\barotimes V}
\circ
a^{-1}_{\{[b]\}\{[b^{-1}]\}\{[a]\}\barotimes V}
\circ
\\
&&
\circ
\id_{\{[b]\}}\barotimes
\Big(
a_{\{[b^{-1}]\}\{[a]\}V}
\circ
\iota_{\{[b^{-1}]\}\barotimes\{[a]\}}^{\{[ab^{-1}]\}}
\barotimes\id_V
\circ
(u*_V v)\Big)
\circ
\\
&&
\circ
a_{\{[b]\}V\{1\}}
\circ
\id_{\{[b]\}\barotimes V}\barotimes\iota^{I_{\VectH}}_{\{1\}}
\circ
r^{-1}_{\{[b]\}\barotimes V}.
\end{eqnarray*}
For $u*_V v$, see
$\eqref{hxalg:equation:*V}$.
\end{lemma}

By taking account of 
Proposition \ref{vecth:proposition:morphismsum},
\eqref{catvecth:eq:projinj},
\eqref{vecth:eq:ipidentity},
and the above lemma,
the first generator
in Definition
\ref{ar:definition:ar}
$(2)$
induces that
$L_{F(\pi)}\circ L_{F(\pi)}^{-1}=\id_{X\barotimes V}$.

Similarly, we obtain
$L_{F(\pi)}^{-1}\circ L_{F(\pi)}=\id_{V\barotimes X}$;
therefore,
$L_{F(\pi)}$ is an isomorphism.
\end{proof}
\begin{proposition}\label{isom:prop:object}
$F(\pi)=(V, L_{F(\pi)})\in \Ob(\Rep \sigma)$.
\end{proposition}
\begin{proof}
We will show \eqref{catlop:eq:RLL=LLR}.
From
Propositions \ref{hxalg:proposition:unique},
\ref{hxalg:prop:product}, \ref{hxalg:prop:point},
and
\eqref{ar:eq:mulmur},
the generators (4) in Definition \ref{ar:definition:ar}
induce that
\begin{eqnarray}\nonumber
&&\sum_{x, y\in X}
(\sigma_{ac}^{xy})_{[ca],[yx]}\barotimes\id_V\circ
(\pi(L_{yd}+I_\sigma)_{[y],[d]}*_V\pi(L_{xb}+I_\sigma)_{[x],[b]})
\\\nonumber
&=&
\sum_{x, y\in X}
(\pi(L_{cy}+I_\sigma)_{[c],[y]}*_V\pi(L_{ax}+I_\sigma)_{[a],[x]})
\circ
\id_V\barotimes
(\sigma_{xy}^{bd})_{[yx],[db]}
\\\label{isom:eq:proof2}
&&
\qquad\qquad\qquad\qquad
\qquad\qquad\qquad\qquad(\forall a, b, c, d\in X).
\end{eqnarray}
With the aid of Proposition \ref{vecth:proposition:morphismsum},
\eqref{hxalg:equation:*V},
and
\eqref{isom:eq:lfpi},
each summand of the left hand side of \eqref{isom:eq:proof2} is
\begin{eqnarray*}
&&(\iota^{\{ a\}\barotimes\{ c\}}_{\{[ca]\}}
\circ
(p_{\{a\}}^X\barotimes p_{\{ c\}}^X))\barotimes\id_V
\circ
\sigma\barotimes\id_V
\circ
\\
&&\circ
((i_X^{\{ x\}}\circ p_{\{ x\}}^X)\barotimes
(i_X^{\{ y\}}\circ p_{\{ y\}}^X))\barotimes\id_V
\circ
a_{X,X,V}^{-1}
\circ
\\
&&\circ
(i_X^{\{ x\}}\circ p_{\{ x\}}^X)\barotimes
((i_X^{\{ y\}}\circ p_{\{ y\}}^X)\barotimes\id_V)
\circ
\id_X\barotimes L_{F(\pi)}
\circ
\\
&&\circ
(i_X^{\{ x\}}\circ p_{\{ x\}}^X)\barotimes
(\id_V\barotimes(i_X^{\{ d\}}\circ p_{\{ d\}}^X))
\circ
a_{X,V,X}
\circ
\\
&&\circ
((i_X^{\{ x\}}\circ p_{\{ x\}}^X)\barotimes
\id_V)\barotimes(i_X^{\{ d\}}\circ p_{\{ d\}}^X)
\circ
L_{F(\pi)}\barotimes\id_X
\circ
\\
&&\circ
(\id_V\barotimes(i_X^{\{ b\}}\circ p_{\{ b\}}^X))\barotimes
(i_X^{\{ d\}}\circ p_{\{ d\}}^X)
\circ
a_{V,X,X}^{-1}
\circ
\id_V\barotimes((i_X^{\{ b\}}\barotimes i_X^{\{ d\}})\circ
\iota_{\{ b\}\barotimes\{ d\}}^{\{[db]\}}).
\end{eqnarray*}

Since the set $X$ is finite,
\eqref{catvecth:eq:projinj},
\eqref{vecth:eq:ipidentity}, and
\eqref{isom:eq:proof2} imply
\eqref{catlop:eq:RLL=LLR}.
\end{proof}
Let $f: \pi_1\to\pi_2$ be a morphism of $\DR(A_\sigma)$.
Because of \eqref{dynrep:eq:defmorphism},
\[
\pi_2(L_{ab}+I_\sigma)\circ m_f
=
m_f\circ
\pi_1(L_{ab}+I_\sigma)
\quad(\forall a, b\in X).
\]
From Lemma \ref{tpdynrep:lem:equivcommute},
this
is equivalent to that
\[
\pi_2(L_{ab}+I_\sigma)_{[a],[b]}
\circ(f\barotimes\id_{\{[b]\}})
=
(\id_{\{[a]\}}\barotimes f)\circ
\pi_1(L_{ab}+I_\sigma)_{[a],[b]}
\quad(\forall a, b\in X),
\]
which immediately induces
\eqref{catlop:eq:morphism};
hence,
\begin{proposition}\label{isom:proposition:morphism}
If $f: \pi_1\to\pi_2$ is a morphism of $\DR(A_\sigma)$,
then
$f: F(\pi_1)\to F(\pi_2)$ is a morphism of $\Rep \sigma$.
\end{proposition}
For $f\in\Hom_{\DR(A_\sigma)}(\pi_1, \pi_2)$,
we define $F(f)\in\Hom_{\Rep \sigma}(F(\pi_1), F(\pi_2))$
by
$F(f)=f$.
A direct computation shows
\begin{proposition}
$F: \DR(A_\sigma)\to\Rep \sigma$ is a strict tensor functor.
\end{proposition}

The next task is to introduce a tensor functor
$G: \Rep \sigma\to \DR(A_\sigma)$.

Let $V\in\Ob(\VectH)$
and
$u: \{\beta^{-1}\}\barotimes V\to V\barotimes\{\alpha^{-1}\}\in
\Hom(\VectH)$
$(\alpha, \beta\in W_X^H)$.
We denote by $u^\wedge: V\barotimes\{\alpha\}\to\{\beta\}
\barotimes V$
the following morphism of $\VectH$.
\begin{eqnarray}\nonumber
u^\wedge
&=&
r_{\{\beta\}\barotimes V}
\circ
\id_{\{\beta\}\barotimes V}
\barotimes
\iota_{I_{\VectH}}^{\{\alpha^{-1}\}\barotimes\{\alpha\}}
\circ
a_{\{\beta\}\barotimes V\{\alpha^{-1}\}\{\alpha\}}
\circ
\\\nonumber
&&
\circ
\Big[
l_{(\{\beta\}\barotimes V)\barotimes\{\alpha^{-1}\}}
\circ
\iota_{I_{\VectH}}^{\{\beta\}\barotimes\{\beta^{-1}\}}
\barotimes\id_{(\{\beta\}\barotimes V)\barotimes\{\alpha^{-1}\}}
\circ
\\\nonumber
&&
\circ
a^{-1}_{\{\beta\}\{\beta^{-1}\}(\{\beta\}\barotimes V)\barotimes\{\alpha^{-1}\}}
\circ
\\\nonumber
&&
\circ
\id_{\{\beta\}}
\barotimes
\Big\{
\id_{\{\beta^{-1}\}}\barotimes
a_{\{\beta\}V\{\alpha^{-1}\}}^{-1}
\circ
a_{\{\beta^{-1}\}\{\beta\}V\barotimes\{\alpha^{-1}\}}
\circ
\\\nonumber
&&
\circ
\iota_{\{\beta^{-1}\}\barotimes\{\beta\}}^{I_{\VectH}}
\barotimes
\id_{V\barotimes\{\alpha^{-1}\}}
\circ
l_{V\barotimes\{\alpha^{-1}\}}^{-1}
\circ
u
\circ
r_{\{\beta^{-1}\}\barotimes V}
\circ
\\\nonumber
&&
\circ
a_{\{\beta^{-1}\}V I_{\VectH}}^{-1}
\circ
\id_{\{\beta^{-1}\}}\barotimes
(\id_V\barotimes
\iota^{\{\alpha\}\barotimes\{\alpha^{-1}\}}_{I_{\VectH}}
\circ
a_{V\{\alpha\}\{\alpha^{-1}\}})\Big\}
\circ
\\\nonumber
&&
\circ
a_{\{\beta\}\{\beta^{-1}\}
(V\barotimes\{\alpha\})\barotimes\{\alpha^{-1}\}}
\circ
\iota_{\{\beta\}\barotimes\{\beta^{-1}\}}^{I_{\VectH}}
\barotimes
\id_{(V\barotimes\{\alpha\})\barotimes\{\alpha^{-1}\}}
\circ
\\\nonumber
&&
\circ
l_{(V\barotimes\{\alpha\})\barotimes\{\alpha^{-1}\}}^{-1}
\Big]
\barotimes\id_{\{\alpha\}}
\circ
\\\label{isom:equation:wedge}
&&
\circ
a_{V\barotimes\{\alpha\}\{\alpha^{-1}\}\{\alpha\}}^{-1}
\circ
\id_{V\barotimes\{\alpha\}}\barotimes
\iota_{\{\alpha^{-1}\}\barotimes\{\alpha\}}^{I_{\VectH}}
\circ
r_{V\barotimes\{\alpha\}}^{-1}.
\end{eqnarray}

A simple verification shows that
$(u^\vee)^\wedge=u$
(For $u^\vee$, see \eqref{isom:eq:vee}).
Moreover, 
the map $\wedge:
\Hom_{\VectH}(\{\beta^{-1}\}\barotimes V,
V\barotimes\{\alpha^{-1}\})
\to
\Hom_{\VectH}(V\barotimes\{\alpha\},
\{\beta\}\barotimes V)$
is an injective $\C$-linear map
because of Proposition \ref{vecth:proposition:morphismsum}
and \eqref{isom:equation:wedge}.
Thus, 
\begin{proposition}
\label{isom:proposition:inverse}
The map $\wedge$ is the inverse of the map
$\vee$.
\end{proposition}

Let $L_V=(V, L_V)$ be an object of $\Rep \sigma$.
For $f, g\in M_H$, define 
$\picheck(f, g)\in D_{H, X}(V)_{1,1}$ by
\[
\picheck(f, g)=\mu_l^{D_{H, X}(V)}(f)\mu_r^{D_{H, X}(V)}(g).
\]
Since
$\picheck$ is $\C$-bilinear,
there uniquely exists
a $\C$-linear map
$\pibar: M_H\otimes_\C M_H\to D_{H, X}(V)$
such that
$\pibar(f\otimes g)=\picheck(f, g)$
$(f, g\in M_H);$
in addition,
$\pibar$ is a $\C$-algebra homomorphism.
We denote by
$\pibar(L_{ab})_{[a][b]}: V\barotimes\{[b]\}
\to\{[a]\}\barotimes V$
and
$\pibar((L^{-1})_{ab})_{[b^{-1}][a^{-1}]}:
V\barotimes\{[a^{-1}]\}\to\{[b^{-1}]\}\barotimes V$
$(a, b\in X)$
the following morphisms of $\VectH$
(cf. \eqref{isom:eq:lfpi}
and
\eqref{isom:eq:lfpi-1}).
\begin{eqnarray*}
&&
\pibar(L_{ab})_{[a][b]}
=
(\iota^{\{ a\}}_{\{[a]\}}\circ p_{\{ a\}}^X)\barotimes\id_V
\circ L_V
\circ\id_V\barotimes (i_X^{\{ b\}}\circ\iota_{\{ b\}}^{\{[b]\}});
\\
&&
\pibar((L^{-1})_{ab})_{[b^{-1}][a^{-1}]}
=
(\id_V\barotimes(\iota^{\{ a\}}_{\{[a]\}}\circ p_{\{ a\}}^X)
\circ L_V^{-1}
\circ(i_{\{ b\}}^X\circ
\iota_{\{ b\}}^{\{[b]\}})\barotimes \id_V)^\wedge.
\end{eqnarray*}

We define the map $\pibar$
from $HX$ \eqref{ar:eq:HX}
to  $D_{H, X}(V)$
as follows:
\begin{eqnarray*}
&&
\pibar(\xi)=\pibar(\xi)\quad(\xi\in M_H\otimes_\C M_H);
\ 
\pibar(L_{ab})=\Gamma_{[a],[b]}^V(\pibar(L_{ab})_{[a][b]});
\\
&&
\pibar((L^{-1})_{ab})=\Gamma_{[b^{-1}],[a^{-1}]}^V
(\pibar((L^{-1})_{ab})_{[b^{-1}],[a^{-1}]}).
\end{eqnarray*}
There consequently exists a unique $\C$-algebra homomorphism
$\pitilde: \C\{ HX\}\to D_{H, X}(V)$
such that
$\pitilde(x)=\pibar(x)$
$(x\in HX)$.
Here, $\C\{ HX\}$ is the free $\C$-algebra on the
set $HX$.
\begin{proposition}\label{isomorph:proposition:pitildeIR}
$\pitilde(I_\sigma)=\{ 0\}$.
\end{proposition}
\begin{proof}
From Lemma \ref{isom:lemma:veeandcircV},
\begin{lemma}
For morphisms
$u: V\barotimes\{[c]\}\to\{[a]\}\barotimes V$
and
$v: V\barotimes\{[c^{-1}]\}\to\{[b^{-1}]\}\barotimes V$
of $\VectH$
$(a, b, c\in X)$,
\begin{eqnarray*}
u*_V v
&=&
l_{\{[ab^{-1}]\}\barotimes V}
\circ
\iota_{I_{\VectH}}^{\{[b^{-1}]\}\barotimes\{[b]\}}
\barotimes
\id_{\{[ab^{-1}]\}\barotimes V}
\circ
a^{-1}_{\{[b^{-1}]\}\{[b]\}\{[ab^{-1}]\}\barotimes V}
\circ
\\
&&
\circ
\id_{\{[b^{-1}]\}}\barotimes
\Big(
\id_{\{[b]\}}\barotimes
(\iota^{\{[b^{-1}]\}\barotimes\{[a]\}}_{\{[ab^{-1}]\}}
\barotimes\id_V
\circ
a^{-1}_{\{[b^{-1}]\}\{[a]\}V})
\circ
\\
&&
\circ
a_{\{[b]\}\{[b^{-1}]\}\{[a]\}\barotimes V}
\circ
\iota_{\{[b]\}\barotimes\{[b^{-1}]\}}^{I_{\VectH}}
\barotimes\id_{\{[a]\}\barotimes V}
\circ
l_{\{[a]\}\barotimes V}^{-1}
\circ
\\
&&
\circ
u\circ
v^\vee
\circ
r_{\{[b]\}\barotimes V}
\circ
\id_{\{[b]\}\barotimes V}\barotimes\iota_{I_{\VectH}}^{\{1\}}
\circ
a_{\{[b]\}V\{1\}}^{-1}
\Big)
\circ
\\
&&
\circ
a_{\{[b^{-1}]\}\{[b]\}V\barotimes\{1\}}
\circ
\iota_{\{[b^{-1}]\}\barotimes\{[b]\}}^{I_{\VectH}}
\barotimes\id_{V\barotimes
\{1\}}
\circ
l^{-1}_{V\barotimes\{1\}}.
\end{eqnarray*}
\end{lemma}
From
\eqref{catvecth:eq:projinj},
Propositions
\ref{vecth:proposition:morphismsum},
\ref{hxalg:prop:product},
\ref{isom:proposition:inverse},
and
the above lemma,
$\pitilde(\sum_{c\in X}L_{ac}(L^{-1})_{cb}-\delta_{ab}\phi)=0$.

Let $u$ denote the generator (4)
of the ideal $I_\sigma$ in Definition \ref{ar:definition:ar}.
On account of the
proof of Proposition
\ref{isom:prop:object},
we can show $\pitilde(u)=0$.
We leave it to the reader to verify that
$\pitilde(u)=0$
for the other generators $u$ of $I_\sigma$.
\end{proof}
From Proposition \ref{isomorph:proposition:pitildeIR},
there uniquely exists
a $\C$-algebra homomorphism
$\pi: A_\sigma\to D_{H, X}(V)$
such that
$\pi(u+I_\sigma)=\pitilde(u)$
$(u\in \C\{ HX\})$.
\begin{proposition}\label{isomorph:proposition:GLV}
$G(L_V):=(V, \pi)\in \Ob(\DR(A_\sigma))$.
\end{proposition}
\begin{proof}
By Definition \ref{hxlag:proposition:hxalghom},
the following two conditions are equivalent.
\begin{enumerate}
\item
$\pi: A_\sigma\to D_{H, X}(V)\in\Hom(\Alg_{(H, X)})$.
\item
$\pi: A_\sigma\to D_{H, X}(V)$ is a $\C$-algebra homomorphism
such that:
\begin{enumerate}
\item
$\pi\circ\mu_l^{A_\sigma}=\mu_l^{D_{H, X}(V)}$,
and
$\pi\circ\mu_r^{A_\sigma}=\mu_r^{D_{H, X}(V)};$
\item
$\pi(L_{ab}+I_\sigma)\in D_{H, X}(V)_{[a],[b]}\quad(\forall a, b\in  X);$
\item
$\pi((L^{-1})_{ab}+I_\sigma)\in D_{H, X}(V)_{[b^{-1}],[a^{-1}]}
\quad(\forall a, b\in  X)$.
\end{enumerate}
\end{enumerate}
We note that, by virtue of \eqref{ar:eq:mulmur},
the above condition (a) in $(2)$ implies
\[
\pi(\xi+I_\sigma)\in D_{H, X}(V)_{1,1}
\quad(\xi\in M_H\otimes_\C M_H).
\]

A direct computation shows $(2)$, and the proposition follows.
\end{proof}
By taking account of the proof of Proposition
\ref{isom:proposition:morphism},
\begin{proposition}
Every $f: L_V\to L_W\in\Hom(\Rep \sigma)$
is a morphism of $\DR(A_\sigma)$ whose source is $G(L_V)$ 
and whose target is
$G(L_W)$.
\end{proposition}
For $f: L_V\to L_W\in\Hom(\Rep \sigma)$, 
we set $G(f)=f: G(L_V)\to G(L_W)\in\Hom(\DR(A_\sigma))$.
\begin{proposition}
$G: \Rep \sigma\to \DR(A_\sigma)$ is a
strict tensor functor.
\end{proposition}

We can check that
$FG=\id_{\Rep \sigma}$ and
$GF=\id_{\DR(A_\sigma)}$
as tensor functors;
therefore,
\begin{theorem}
The tensor categories $\Rep \sigma$ and
$\DR(A_\sigma)$ are isomorphic.
\end{theorem}
\begin{corollary}\label{isomorph:theorem:theorem}
$\Rep R$
is isomorphic to $DR(A_R)$.
\end{corollary}
\begin{remark}\label{isomorph:remark:same}
If $\sigma$ is a Yang-Baxter operator on $X$ in $\VectH$
(Definition \ref{catlop:definition:YBop}),
then
the objects $(X, \sigma)$ of $\Rep \sigma$
and $(X, \pi_\sigma)$
of $DR(A_\sigma)$
satisfy
$G(X, \sigma)=(X, \pi_\sigma)$
(See Propositions \ref{catlop:proposition:equivYBop},
\ref{lop:proposition:X},
and
\ref{dynrep:proposition:DHXXpiR}).
\end{remark}
}
\section*{Acknowledgements}
The first author wishes to express his thanks to Professor
Mikhail Olshanetsky for drawing his attention to
L-operators for the dynamical Yang-Baxter map.
The first and the second authors were supported in part by KAKENHI
(19540001)
and (18540008), respectively.


\begin{thebibliography}{99}
\bibitem{baxter1972}
Baxter, R.J.:
Partition function of the eight-vertex lattice model,
{\em Ann. Physics}
{\bf 70} (1972)
193--228.
\bibitem{baxter1982}
Baxter, R.J.:
\textit{Exactly solved models in statistical mechanics\/}.
Academic Press, Inc., London, 1982.
\bibitem{bohm}
B\"{o}hm, G.:
Hopf algebroids,
arXiv:0805.3806v1\ [math.QA], 2008.
\bibitem{brzezinski}
Brzezi\'{n}ski, T. and
Militaru, G.:
Bialgebroids, $\times_A$-bialgebras and duality,
{\em J. Algebra}
{\bf 251} (2002), no.\ 1,
279--294.
\bibitem{donin}
Donin, J. and
Mudrov, A.:
Dynamical Yang-Baxter equation and quantum vector bundles,
{\em Comm. Math. Phys.} 
{\bf 254} (2005), no.\ 3,
719--760.
\bibitem{drinfeld}
Drinfel'd, V.G.:
On some unsolved problems in quantum group theory,
\textit{Quantum groups}
(Leningrad, 1990), 1--8, Lecture Notes in Math., 1510, Springer, Berlin, 1992. 
\bibitem{etingof2005}
Etingof, P. and
Latour, F.:
\textit{The dynamical Yang-Baxter equation, representation theory,
and quantum integrable systems\/}.
Oxford Lecture Series in Mathematics and its Applications, 29. 
Oxford University Press, Oxford, 2005.
\bibitem{etingof1999}
Etingof, P., Schedler, T. and Soloviev, A.:
Set-theoretical solutions of the quantum Yang-Baxter equation,
{\em Duke Math. J.}
{\bf 100} (1999), no.\ 2, 169--209.
\bibitem{etingof1998}
Etingof, P. and Varchenko, A.:
Solutions of the quantum dynamical Yang-Baxter equation and
dynamical quantum groups,
{\em Commun. Math. Phys.}
{\bf 196} (1998), no.\ 3, 591--640.
\bibitem{faddeev}
Faddeev, L.D., Reshetikhin, N.Yu.,
and
Takhtajan, L.A.:
Quantization of Lie groups and Lie algebras,
\textit{Algebraic analysis}, Vol. I, 129--139, Academic Press, Boston, MA, 1988.
\bibitem{felder}
Felder, G.:
Conformal field theory and integrable systems
associated to elliptic curves,
\textit{Proceedings of the International Congress of Mathematicians},
Vol. 1, 2 (Z\"{u}rich, 1994), 1247--1255, Birkh\"{a}user, Basel, 1995.
\bibitem{gervais}
Gervais, J.-L.
and
Neveu, A:
Novel triangle relation and absence of tachions in Liouville string field theory,
{\em Nuclear Phys. B} {\bf 238}
(1984), no.\ 1,
125--141.
\bibitem{hartwig}
Hartwig, J.T.:
The elliptic $GL(n)$ dynamical quantum group as
an $\mathfrak h$-Hopf algebroid,
{\em Int. J. Math. Math. Sci.} {\bf 2009},
Art. ID 545892, 41 pp.
\bibitem{kassel}
Kassel, C.:
\textit{Quantum groups\/}.
Graduate Texts in Mathematics, 155. Springer-Verlag, New York, 1995.
\bibitem{koelink2007}
Koelink, E. and van Norden, Y.:
Pairings and actions for dynamical quantum groups,
{\em Adv. Math.}
{\bf 208} (2007), no.\ 1, 1--39.
\bibitem{lu}
Lu, J.-H.:
Hopf algebroids and quantum groupoids,
{\em Internat. J. Math.}
{\bf 7}
(1996), no.\ 1,
47--70.
\bibitem{lu2000}
Lu, J.-H., Yan, M. and Zhu, Y.-C.:
On the set-theoretical Yang-Baxter equation,
{\em Duke Math. J.} {\bf 104}  (2000),  no. 1, 1--18.
\bibitem{pflugfelder}
Pflugfelder, H.O.:
\textit{Quasigroups and loops\/$:$ Introduction}.
Sigma Series in Pure Mathematics, 7.\ Heldermann Verlag, Berlin, 1990.
\bibitem{shibukawa05}
Shibukawa, Y.:
Dynamical Yang-Baxter maps,
{\em Int. Math. Res. Not.}
{\bf 2005}, no\ 36,
2199--2221.
\bibitem{shibukawa07}
Shibukawa, Y.:
Dynamical Yang-Baxter maps with an invariance condition,
{\em Publ. Res. Inst. Math. Sci.}
{\bf 43}, no.\ 4,
(2007)
1157--1182.
\bibitem{takeuchi}
Takeuchi, M.:
Groups of algebras over
$A\otimes \bar{A}$,
{\em J. Math. Soc. Japan}
{\bf 29} (1977), no.\ 3,
459--492.
\bibitem{takeuchi2003}
Takeuchi, M.:
Survey on matched pairs of groups--an elementary approach
to the ESS-LYZ theory,
{\em Noncommutative geometry and quantum groups
$($Warsaw, $2001)$}, 305--331,
Banach Center Publ., 61, Polish Acad. Sci., Warsaw, 2003. 
\bibitem{takeuchi2005}
Takeuchi, M.:
A survey on Nichols algebras,
{\em Contemp. Math.}
{\bf 376}
(2005)
105--117.
\bibitem{veselov}
Veselov, A.:
Yang-Baxter maps:
dynamical point of view,
\textit{Combinatorial Aspect of Integrable Systems},
145--167, MSJ Mem., 17, Math. Soc. Japan, Tokyo, 2007.
\bibitem{weinstein}
Weinstein, A. and Xu, P.:
Classical solutions of the quantum Yang-Baxter equation,
{\em Comm. Math. Phys.} {\bf 148}  (1992),  no. 2, 309--343.
\bibitem{xu}
Xu, P.:
Quantum groupoids,
{\em Commun. Math. Phys.}
{\bf 216} (2001), no.\ 3,
539--581.
\bibitem{yang1967}
Yang, C.N.:
Some exact results for the many-body problem
in one dimension with repulsive delta-function,
{\em Phys. Rev. Lett.}
{\bf 19} (1967),
1312--1314.
\bibitem{yang1968}
Yang, C.N.:
S matrix for the one-dimensional N-body problem
with repulsive or attractive $\delta$-function,
{\em Phys. Rev.}
{\bf 168} (1968),
1920--1923.
\end{thebibliography}
\end{document}